\documentclass[12pt]{amsart}
\renewcommand{\bar}{\overline}
\usepackage{amsmath,amsthm,amscd,euscript,longtable}
\renewcommand{\bar}{\overline}
\def \r{\mathbb R}

\def \z{\mathbb Z}

\def \D{\mathcal D_\gamma}
\def \Dd{\mathcal D}

\def \({\langle}
\def \){\rangle}

 \makeatletter
\renewenvironment{subequations}{%
  \protected@edef\theparentequation{\theequation}%
  \setcounter{parentequation}{\value{equation}}%
  \setcounter{equation}{0}%
  \def\theequation{\roman{equation}}%
  \ignorespaces
}{%
  \setcounter{equation}{\value{parentequation}}%
  \ignorespacesafterend
} \makeatother

\newtheorem{theorem}{Theorem}[section]
\newtheorem{lemma}[theorem]{Lemma}

\newtheorem{proposition}[theorem]{Proposition}
\newtheorem{corollary}[theorem]{Corollary}
\theoremstyle{remark}
\newtheorem{remark}[theorem]{Remark}
\theoremstyle{definition}
\newtheorem{definition}[theorem]{Definition}
\newtheorem{example}[theorem]{Example}
\newtheorem{problem}{Problem}

\title{Finite and infinitesimal flexibility of semidiscrete surfaces}
\author{Oleg Karpenkov}
\date{17 April 2015}

\thanks{The work is partially supported by FWF grant No.~S09209.}


\email[Oleg Karpenkov]{karpenk@liv.ac.uk}
\address{University of Liverpool}

\begin{document}
\input{epsf}

\begin{abstract}
In this paper we study infinitesimal and finite flexibility for
generic semi\-discrete surfaces.
We prove that generic 2-ribbon semidiscrete surfaces have one degree of infinitesimal and finite flexibility.
In particular we write down a system of differential equations describing isometric deformations
in the case of existence.
Further we find a necessary condition of 3-ribbon infinitesimal flexibility.
For an arbitrary $n\ge 3$ we prove that every generic $n$-ribbon surface has at most one degree of finite/infinitesimal flexibility.
Finally, we discuss the relation between general semidiscrete surface flexibility and 3-ribbon subsurface flexibility.
We conclude this paper with one surprising property of isometric
deformations of developable semidiscrete surfaces.
\end{abstract}

\maketitle

\tableofcontents

\section*{Introduction}

A mapping $f:\r\times \z \to\r^3$, where the dependence on the
continuous parameter is smooth, is called a {\it semidiscrete
surface}. Let us connect $f(t,z)$ with $f(t,z{+}1)$ by segments
for all possible pairs $(t,z)$. The resulting
surface is a {\it piecewise ruled surface}.

In this paper we study infinitesimal and finite flexibility
for such semidiscrete surfaces. By {\it isometric deformations} of a semidiscrete surface
$f$ we understand deformations that preserves inner geometry of the
corresponding ruled surfaces and in addition that preserve all
line segments connecting $f(t,z)$ with $f(t,z{+}1)$.

\vspace{2mm}

Many questions on discrete polyhedral surfaces have their origins
in classical theory of smooth surfaces. Flexibility is not an
exception from this rule. The general theory of flexibility of
surfaces and polyhedra is discussed in the overview~\cite{Sab} by
I.~Kh.~Sabitov.

In 1890~\cite{Bia2} L.~Bianchi introduced a necessary and
sufficient condition for the existence of isometric deformations
of a surface preserving some conjugate system (i.e., two
independent smooth fields of directions tangent to the surface),
see also in~\cite{Eis}. Such surfaces can be understood as certain
limits of semidiscrete surfaces.

On the other hand, semidiscrete surfaces are themselves the limits
of certain polygonal surfaces (or {\it meshes}). For the discrete
case of flexible meshes much is now known. We refer the reader
to~\cite{BHS}, ~\cite{WP}, ~\cite{Kok}, and~\cite{3x3} for some
recent results in this area. For general relations to the
classical case see a recent book~\cite{BS} by A.~I.~Bobenko and
Yu.~B.~Suris. It is interesting to notice that the flexibility
conditions in the smooth case and the discrete case are of a
different nature. Currently there is no clear description of
relations between them in terms of limits.

The place of the study of semidiscrete surfaces is between the
classical and the discrete cases. Main concepts of semidiscrete
theory are described by J.~Wallner in~\cite{Wal1},
and~\cite{Wal2}. Some problems related to isothermic semidiscrete
surfaces are studied by C.~M\"uller in~\cite{Mue}.
Semidiscrete surfaces from the viewpoint of parallelity,
offsets, and curvatures were studied by J.~Wallner and O.~Karpenkov in~\cite{KW2014}.

\vspace{2mm}

We investigate necessary condition for existence of isometric
deformations of semidiscrete surfaces. To avoid pathological
behavior related to noncompactness of semidiscrete surfaces we
restrict ourselves to compact subsets of the following type. An
{\it $n$-ribbon
 surface} is a mapping
$$
f:[a,b]\times \{0,\ldots, n\} \to\r^3, \qquad (t,i) \mapsto
f_i(t).
$$
We also use the notion
$$
\Delta f_i(t)=f_{i+1}(t)-f_{i}(t). $$

\vspace{2mm}

While working with a rather abstract semidiscrete or $n$-ribbon
surface $f$ we keep in mind the two-dimensional piecewise-ruled
surface associated to it (see Fig.~\ref{part}).

\begin{figure}
$$\epsfbox{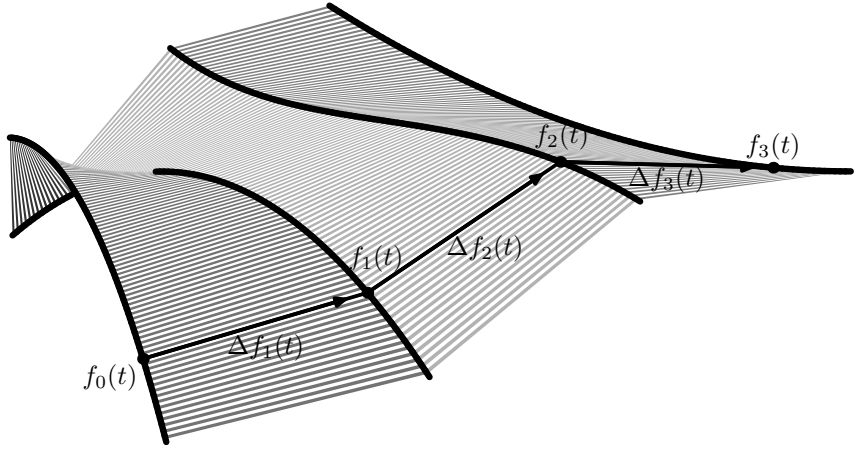}$$
\caption{A  3-ribbon surface.}\label{part}
\end{figure}

\vspace{2mm}

Note that, within this paper we traditionally consider $t$ as an argument of a semidiscrete
surface $f$. The time parameter for deformations is $\lambda$.

\vspace{2mm}

In present paper we prove that every generic 2-ribbon  surface (as a ruled
surface) is flexible and has one degree of infinitesimal and finite flexibility in the generic
case (Theorem~\ref{1-d-inf} and Theorem~\ref{2-ribbon flex}). This is quite surprising since
generic 1-ribbon surfaces have infinitely many degrees of flexibility,
see, for instance, in~\cite{PW2}, Theorem~5.3.10.  We also find a
system of differential equations for the deformation of 2-ribbon
surfaces (Definition~\ref{defODE} and Proposition~\ref{euV+}). In contrast to that,
a generic $n$-ribbon surface is rigid for $n\ge 3$. For the case
$n=3$ we prove the following statement (see Theorem~\ref{inf} and
Remark~\ref{infREM}).

\vspace{1mm}

{
\noindent
{\bf Infinitesimal flexibility condition.}\\ {\it If a 3-ribbon
surface is infinitesimally flexible then the following
condition holds:
$$
\dot{\Lambda}=(H_2-H_1)\Lambda,
$$
where
$$
\Lambda= \frac{\big(\dot f_1,\ddot f_1,\Delta f_0\big)}{\big(\dot
f_2,\ddot f_2,\Delta f_2\big)} \frac{\big(\dot f_2,\Delta
f_1,\Delta f_2\big)^2}{\big(\dot f_1,\Delta f_0,\Delta
f_1\big)^2},
$$
and
$$
H_i(t)= \frac{(\dot f_i, \Delta \dot f_{i-1},\Delta f_{i})
+(\dot{f}_i, \Delta f_{i-1}, \Delta \dot f_{i})}
{(\dot{f}_i,\Delta f_{i-1},\Delta f_{i})}, \quad i=1,2.
$$
}
{\it Remark.} Throughout this paper we denote the derivative with
respect to variable $t$ by the dot symbol.
}

\vspace{2mm}


Further in Theorem~\ref{proportionalND} we state that a generic $n$-ribbon surface ($n\ge 3$)
has at most one degree of finite and infinitesimal flexibility.
Finally, we
show that a generic $n$-ribbon surface ($n\ge 4$) is infinitesimally or finitely
flexible if and only if all its 3-ribbon subsurfaces are
infinitesimally or finitely flexible (see
Theorems~\ref{n-ribbon-infini} and~\ref{n-ribbon-fini}). We say a
few words in the case of developable semidiscrete surfaces whose
finite isometric deformations have additional surprising properties.

\vspace{2mm}

{\bf Organization of the paper.}
We start in Section~1 with introduction of necessary notions and definitions.
In Section~2 we discuss flexibility of 2-ribbon surfaces.
We study infinitesimal flexibility questions for 2-ribbon surfaces in Subsections~2.2
and~2.3.
In Subsection~2.2 we give a system of differential equations for infinitesimal flexions,
prove the existence of nonzero solutions,
and show that all the solutions are proportional to each other.
In Subsection~2.3 we define the variational operators of infinitesimal flexion which is studied
further in the context of finite flexibility for 2-ribbon surfaces.
In Subsection~2.4 we prove that a generic 2-ribbon surface is
finitely flexible and has one degree of flexibility.
In Section~3 we work with $3$-ribbon surfaces.
After some preliminary statements of Subsection~3.1 we gives a necessary infinitesimal flexibility condition for
3-ribbon surfaces in Subsection~3.2.
In Section 4 we deal with general $n$-ribbon surfaces for $n\ge 3$.
We prove that a generic $n$-ribbon surface
has at most one degree of finite and infinitesimal flexibility in Subsection~4.1.
Further after several preparatory statements of Section~4.2 we prove that
finite or infinitesimal flexibility of generic $n$-ribbon surfaces is identified by
finite or infinitesimal flexibility of all its 3-ribbon subsurfaces.
We conclude the paper with flexibility of developable semidiscrete surfaces in
Section~5. In this case isometric deformations have a remarkable geometric property.

\section{Necessary notions and definitions}

In this section we introduce central notions and definition of the article.

\subsection{Differentiable generic semidiscrete surfaces}
We start with several basic definitions.

\begin{definition}
Let $M=(m_0,\ldots, m_n)$ be the $(n{+}1)$-tuple of non-negative integers.
We say that an $n$-ribbon surface $f$ is a {\it $M$-differentiable} if
for every $i\in\{0,\ldots, n\}$ and $j\in\{1,\ldots, m_i\}$
there exists a continuous derivative $f^{(j)}_i$.
\\
Denote by $C^{m_0,\ldots,m_n}([a,b],\r^3)$ (or $C^M([a,b],\r^3)$, for short) the Banach space of all $M$-differentiable $n$-ribbon surfaces (where $t\in[a,b]$)
with the standard norm
$$
\rho(f,g)=
\max\limits_{i=\{0,\ldots, n\}}
\max\limits_{j=\{1,\ldots, m_i\}}
\sup\limits_{[a,b]}(f^{(j)}_i-g^{(j)}_i).
$$
\end{definition}

\begin{remark}
Note that for two non-negative $(n{+}1)$-tuples $M=(m_0,\ldots, m_n)$ and $K=(k_0,\ldots, k_n)$ satisfying
$$
m_0\ge k_0, \quad \ldots, \quad m_n\ge k_n
$$
in holds
$$
C^{M}([a,b],\r^3)
\subset
C^{K}([a,b],\r^3).
$$
\end{remark}

\begin{definition}
We say that an $n$-ribbon surface $f$ in the space $C^{1,2,2,\ldots, 2,1}([a,b],\r^3)$
is {\it weakly generic} if
for every $t\in [a,b]$ and $i=1,\ldots,n-1$ we have
$$
(\dot f_i,\Delta f_{i-1},\Delta f_{i})\ne 0.
$$
\end{definition}

\begin{definition}
We say that an $n$-ribbon surface $f$ in the space $C^{1,2,2,\ldots, 2,1}([a,b],\r^3)$
is {\it strongly generic} if

--- $f$ is weakly generic;

--- for every $t\in [a,b]$ and $i=1,\ldots,n-1$ we have
$$
\big(\dot f_i(t),\ddot f_i(t),\Delta f_{i-1}(t)\big)\ne 0
\qquad \hbox{and} \qquad
\big(\dot f_{i}(t),\ddot f_{i}(t),\Delta f_{i}(t)\big) \ne 0.
$$
\end{definition}

\subsection{Isometric semidiscrete surfaces}
Let us now study basic properties of the definition of isometric semidiscrete surfaces.

\begin{definition}\label{isometric_conditions}
Two $n$-ribbon surfaces $f$ and $g$ in the space $C^{1,1,\ldots, 1}([a,b],\r^3)$ are said to be {\it isometric} if
$$
\left\{
\begin{array}{l}
|\dot f_i|=|\dot g_i|\\
|\Delta f_{i}|=|\Delta g_{i}|\\
\(\dot f_i,\Delta f_{i-1}\)=\(\dot g_i,\Delta g_{i-1}\)\\
\(\dot f_i,\Delta f_{i}\)=\(\dot g_i,\Delta g_{i}\)\\
\(\dot f_i,\dot f_{i+1}\)=\(\dot g_i,\dot g_{i+1}\)
\end{array}
\right.
$$
(for all admissible $i$ and $t$).
\end{definition}

Before to continue let us show that the conditions of Definition~\ref{isometric_conditions}
are precisely the isometric conditions for ruled surfaces.
Let $f_1$ and $f_2$ be differentiable curves (denote by $\Delta_1 f$ the curve $f_2{-}f_1$).
Let us define a ruled surface
$S(x,t)=xf_1(t)+ (1{-}x)f_2(t)$.
To show that the conditions of Definition~\ref{isometric_conditions}
determine integer geometry we prove the following proposition.

\begin{proposition}
The first fundamental form of the ruled surface $S(x,t)$
is uniquely defined by
$$
|\dot f_1|, \quad |\dot f_2|, \quad |\Delta f_1|, \quad  \(\dot f_1,\Delta f_1\),\quad
\(\dot f_2,\Delta f_1\),\quad  \(\dot f_1,\dot f_2\)
$$
and vice versa.
\end{proposition}

\begin{proof}
Let us write all the coefficients of the first fundamental form of the surfaces in the coordinates $(x,t)$:
$$
\begin{array}{l}
\Big\(\frac{\partial S}{\partial x},\frac{\partial S}{\partial x}\Big\)=\(f_1{-}f_2,f_1{-}f_2\)=|\Delta f_1|^2;
\\
\Big\(\frac{\partial S}{\partial x},\frac{\partial S}{\partial t}\Big\)
=
\big\(f_1{-}f_2,x\dot f_1{+} (1{-}x)\dot f_2\big\)=
x\(\Delta f_1,\dot f_1\)+ (1{-}x)\(\Delta f_1,\dot f_2\);
\\
\Big\(\frac{\partial S}{\partial t},\frac{\partial S}{\partial t}\Big\)=
\big\(x\dot f_1{+} (1{-}x)\dot f_2(t),x\dot f_1{+} (1{-}x)\dot f_2(t)\big\)
\\
\qquad\qquad\hbox{ }
=x^2|f_1|^2+2x(1{-}x)\(\dot f_1,\dot f_2\)+(1{-}x)^2|f_2|^2.
\end{array}
$$
As we see, on the one hand the first fundamental form is defined by the above six functions.
On the other hand the values of the first fundamental form at $x=0,1/2,1$ defines the values of the above six functions.
\end{proof}

\subsection{Deformations and flexions of semidiscrete surfaces}
We start with the following general definition.
\begin{definition}
A {\it deformation} of a semidiscrete $n$-ribbon surface $f$ is a family
of $n$-ribbon surfaces $\{f^\lambda\}$ with parameter $\lambda$ in the segment $[-\Lambda,\Lambda]$
for some positive $\Lambda$ such that $f^0=f$.
In this paper we consider only deformations that are continuously differentiable in $\lambda$.
\end{definition}

\begin{remark}
In this paper $\lambda $ is the parameter of deformations, while
$t$ is the first argument of semidiscrete surfaces.
\end{remark}

%
%
%

Let us give a formal definition of deformations that do not change the inner geometry of a surface.

\begin{definition}
We say that a deformation $\{f^\lambda\}$ of a semidiscrete $n$-ribbon surface $f$ is {\it isometric} if
all the surfaces in the deformation are isometric to each other.
\end{definition}

\begin{definition}
Consider a family of functions, vector functions, or semidiscrete surfaces  $\gamma=\{w^\lambda\}$ with parameter $\lambda\in [-\varepsilon,\varepsilon]$
for some positive $\varepsilon$, and let
$w=w^0$.
We say that the derivative
$$
\D w=\frac{\partial w^\lambda}{\partial \lambda}\Big|_{\lambda=0}
$$
is an {\it infinitesimal deformation} of $w$.
\end{definition}

The infinitesimal deformation of an $n$-ribbon surface $f$ in $C^M([a,b],\r^3)$ is an element of the tangent space $T_fC^M([a,b],\r^3)$,
which is naturally isomorphic to $C^M([a,b],\r^3)$.

\begin{definition}
Consider a deformation $\{f^\lambda\}$ of a semidiscrete $n$-ribbon surface $f$ in $C^{(1,2,2,\ldots, 2,1)}([a,b],\r^3)$.
We say that the deformation $\{f^\lambda\}$ is {\it infinitesimally flexible} if
$$
\begin{array}{c}
\D|\dot f_i^\lambda|=0, \qquad  \D|\Delta f_{i}^\lambda|=0,  \qquad \D\(\dot f_i^\lambda,\Delta f_{i-1}^\lambda\),
\\
\D\(\dot f_i^\lambda,\Delta f_{i}^\lambda\)=0, \quad \hbox{and} \quad \D\(\dot f_i^\lambda,\dot f_{i+1}^\lambda\)=0
\end{array}
$$
(for all admissible $i$ and $t$).
\end{definition}

In fact, infinitesimal flexibility is a property of tangent spaces rather than deformations.

\begin{definition}
We say that a tangent vector $\Dd f$ at a semidiscrete surface $f$ is an {\it infinitesimal flexion}
if the deformation $\D f$ where
$$
\gamma(\lambda)=f+\lambda \Dd f
$$
is infinitesimally isometric.
\\
We say that an infinitesimal flexion $\Dd f$ is a {\it finite flexion} if there exists an isometric deformation $\gamma$
with $\gamma(0)=f$ such that $\Dd_{\gamma} f=\Dd f$.
\end{definition}

Finally let us determine isometrically nontrivial infinitesimal flexions.

\begin{definition}\label{def_isom}
An infinitesimal flexion of a weakly generic $n$-ribbon surface $f$ in $C^{0,1,0}([a,b],\r^3)$ is said to be {\it isometrically nontrivial $($trivial$)$ at point
$(t,i)$} for some $t\in[a,b]$ and $n\in\{1,\ldots,n-1\}$ if
the corresponding infinitesimal deformation of the angle between the planes spanned by $(\dot f_i(t)\Delta f_{i-1}(t))$ and
$(\dot f_i(t)\Delta f_{i}(t))$ is nonzero (or zero, respectively).
\\
We say that an infinitesimal flexion of $f$ is {\it isometrically nontrivial} if it is isometrically nontrivial
at least at one point $(t,i)$. Otherwise an infinitesimal inflexion is said to be {\it isometrically trivial}.
\\
We say that an infinitesimal flexion of $f$ is {\it strongly isometrically nontrivial} if it is isometrically nontrivial
at every point $(t,i)$.
\end{definition}

\subsection{Spaces of semidiscrete surfaces with fixed initial position}
In order to calculate the degree of flexibility for a semidiscrete surfaces we should eliminate trivial Euclidean deformations of the surfaces.
Let us do this as follows.

\begin{definition}
Denote by
$$
C_0^M([a,b],\r^3)\subset C^M([a,b],\r^3)
$$
the subset of all {\it 2-ribbon surfaces with fixed initial position},
namely an $n$-ribbon surface $f$ is in $C_0^M([a,b],\r^3)$ if and only if

--- $f_1(0)\in C^M([a,b],\r^3)$;

--- $f_1(0)=(0,0,0)$;

--- the vector $\dot f_1(0)$ is proportional to $(1,0,0)$;

--- the vector $\Delta f_0(0)$ has the coordinates $(p,q,0)$.
\end{definition}

\begin{remark}
Let $\Sigma$ denotes all weakly non-generic semidiscrete surfaces.
Notice that the set $C_0^M([a,b],\r^3)\setminus \Sigma$ has a natural structure of an 8-fold covering of the
quotient space of $C^M([a,b],\r^3)\setminus \Sigma$ by the Euclidean congruence relation.
In other words, for every weakly generic $M$-differentiable semidiscrete surface $f$
there exists exactly eight semidiscrete surfaces that are congruent to $f$.
These 8 surfaces are obtained one from another by 8 symmetries of type
$$
(e_1,e_2,e_3)\to (\pm e_1,\pm e_2, \pm e_3).
$$
So, on the one hand one can consider any branch of the 8-fold for studying flexibility properties of the original $n$-ribbon curve.
On the other hand the set $C_0^M([a,b],\r^3)$ has a structure of a vector space.
For these reasons from now on we prefer to consider the space $C_0^M([a,b],\r^3)$,
rather than the quotient space of $C^M([a,b],\r^3)\setminus \Sigma$ by the group of all Euclidean transformation.
\end{remark}

Since $C_0^M([a,b],\r^3)$ is a subspace of $C^M([a,b],\r^3)$ we have the induced metric and topology
(in particular, $C_0^M([a,b],\r^3)$ is a Banach space),
definitions of deformations, isometric deformations, infinitesimal and finite flexions,
isometrically trivial and nontrivial infinitesimal flexions in $C^M_0([a,b],\r^3)$.

\subsection{Rigid surfaces. Degrees of flexibility}

We start with the definitions for infinitesimal flexibility.

\begin{definition}
The set of infinitesimal flexions in $C^M_0([a,b],\r^3)$ is a linear space.
We say that $f$ has {\it $n$ degrees of infinitesimal flexibility}
if the dimension of the space of infinitesimal flexions is $n$.
If $n=0$ we say that $f$ is {\it infinitesimally rigid}.
\end{definition}

In the finite case we define only finitely rigid surfaces and surfaces that has one degree of finite flexibility.
In order to define finite rigidity we use the following definition.

\begin{definition}
We say that an isometric deformation $\gamma$ of $f$ in $C^M_0([a,b],\r^3)$ is {\it regular} at $0$ if
$\Dd_{\gamma} f\ne 0$.
\end{definition}

\begin{definition}
We say that an $n$-ribbon surface $f$ in $C^M_0([a,b],\r^3)$ is {\it finitely rigid} if
the set of regular isometric deformations of $f$ is empty.
\end{definition}

Let us finally give the definition of the property to have one degree of finite flexibility.
As in infinitesimal case we consider only the space of semidiscrete surfaces with fixed initial position $C_0^M([a,b],\r^3)$.
This cancels excess trivial Euclidean rotations of the whole semidiscrete surface.
Of course, every finite isometric deformations of a semidiscrete surface with fixed initial position still can be reparametrised,
as a result one has another isometric deformation of the surface.
So the best thing would be to try to normalize them.

In this paper we consider the following ``natural parametrization'' of an isometric deformation.
It is clear that for every isometric deformation $\{f^\lambda\}$ in $C_0^M([a,b],\r^3)$
we have
$$
\Dd_{f^\lambda} \dot f(a)=0, \quad \Dd_{f^\lambda} \Delta f_0(a)=0,
\quad \hbox{and} \quad
\Dd_{f^\lambda} \Delta f_1(a)= \alpha(\lambda) \dot f(a) {\times} \Delta f_1(a)
$$
for some real valued function $\alpha$.

\begin{definition}
We say that an isometric deformation $\{f^{\lambda}\}$ is {\it normalized} if and only if
for every admissible values of parameter $\lambda$ we have $\alpha(\lambda)=1$, where $\alpha$ is
the real-valued function defined in the last expression.
\end{definition}

In our case by Corollary~\ref{000} below we have: if $\alpha(\lambda_0)=0$ then $\Dd_{f^\lambda} f^{\lambda_0}=0$.
Hence, there is no regular isometric deformation that preserves the frame at $t=a$.
So we can give the following definition.

\begin{definition}
We say that a weakly generic 2-ribbon surface $f$ has {\it one degree of finite flexibility}
if

--- $f$ has one degree of infinitesimal flexibility.

--- for sufficiently small $\varepsilon>0$ there exists a unique normalized isometric deformation of $f$ defined on $[-\varepsilon, \varepsilon]$.
\end{definition}

\section{Finite and infinitesimal flexibility of 2-ribbon surfaces}

In this section we describe flexions of 2-ribbon surfaces. Such
surfaces are defined by three curves $f_0$, $f_1$, and $f_2$. Our
main goal here is to prove under some natural genericity
assumptions that every 2-ribbon surface is infinitesimally and finitely flexible and has one
degree of infinitesimal and finite flexibility. Our first point is to describe the system of
differential equations (System~A) that determines infinitesimal
flexions corresponding to finite flexions and find solutions to
this system  (see Subsections~\ref{inf_flex_subs}).
We use it to derive finite flexibility in Theorem~\ref{1-d-inf} (also in Subsections~\ref{inf_flex_subs}).
Further via solutions of System~A we define the variational operators of infinitesimal
flexion ${\mathcal V}^{\pm}$ (in Subsection~\ref{var}). Finally, to show
finite flexibility of 2-ribbon surfaces we study Lipschitz
properties for ${\mathcal V}^{\pm}$ and prove flexibility Theorem~\ref{2-ribbon flex} (in Subsection~\ref{finit2r}).

\subsection{Basic relations for infinitesimal flexions}\label{basRel}

In this small subsection we collect some useful relations.

\begin{proposition}
Let $f$ be a 2-ribbon surface in $C^{1,2,1}([a,b],\r^3)$.
Then for every infinitesimal flexion $\Dd f$ the
following properties hold:
\begin{eqnarray}
&&\(\dot{f}_1,\Dd\dot{f}_1\)=0;\label{e1}
\\
&&\(\dot{f}_1-\Delta\dot{f}_0,\Dd\dot{
f}_1-\Dd\Delta\dot{f}_0\)=0;\label{e2}
\\
&&\(\dot{f}_1+\Delta\dot{f}_1,\Dd\dot{
f}_1+\Dd\Delta\dot{f}_1\)=0;\label{e3} \\
&&\(\Delta
{f}_0,\Dd\Delta\dot{f}_0\)+\(\Delta\dot{f}_0,\Dd \Delta
f_0\)=0;\label{e4}
\\
&&\(\Delta {f}_1,\Dd\Delta\dot{f}_1\)+\(\Delta\dot{f}_1,\Dd \Delta
f_1\)=0;\label{e5}
\\
&&\(\dot {f}_1,\Dd\Delta\dot{f}_0\)+\(\Dd\dot f_1,\Delta \dot
f_0\)=0;\label{e8}
\\
&&\(\dot {f}_1,\Dd\Delta\dot{f}_1\)+\(\Dd \dot f_1, \Delta\dot
f_1\)=0;\label{e9}
\\
&&\(\Dd\ddot{f}_1,\Delta {f}_0\)+\(\ddot{f}_1,\Dd \Delta
{f}_0\)=0;\label{e6}
\\
&&\(\Dd\ddot{f}_1,\Delta {f}_1\)+\(\ddot{f}_1,\Dd \Delta
{f}_1\)=0.\label{e7}
\end{eqnarray}
\end{proposition}

\begin{remark}
For a semidiscrete or $n$-ribbon surface $f$ the operations $\Dd$, $\Delta$, and
$\frac{\partial}{\partial t}$ commute, so we do not pay attention
to the order of these operations in compositions.
\end{remark}

\begin{proof}
Equations~(\ref{e1}), (\ref{e2}), and~(\ref{e3})
follow from the fact that infinitesimal
flexions preserve the norms of  $\dot f_1$,
$\dot f_{0}=\dot f_1-\Delta\dot f_0$, and $\dot f_2=\dot f_1+\Delta\dot f_1$ respectively.

The invariance of the lengths of $\Delta f_0$ and $\Delta f_1$
imply Equations~(\ref{e4}), and~(\ref{e5}) respectively. They are equivalent to
$$
 \frac{\partial}{\partial t}\Dd\(\Delta f_0,\Delta f_0\)=0
\quad \hbox{and} \quad
\frac{\partial}{\partial t}\Dd\(\Delta f_1,\Delta f_1\)=0.
$$

Equations~(\ref{e8}) and~(\ref{e9}) follow from invariance of
the angles between the vectors $\dot {f_1}$ and $\Delta \dot {f_{0}}$
and the vectors $\dot {f_1}$ and $\Delta \dot {f_{0}}$.

Let us prove Equation~(\ref{e6}).
Since the angles between the vectors
$\Delta f_0$ and $\dot{f_1}$ are
preserved by infinitesimal flexions we have
$$
\frac{\partial}{\partial t}\Dd\(\dot{f}_1, \Delta f_0\)=0.
$$
Therefore,
$$
\(\Dd\ddot{f}_1, \Delta f_0\)+\(\ddot{f}_1, \Dd\Delta f_0\)+
\(\Dd\dot{f}_1, \Delta \dot f_0\)+\(\dot{f}_1, \Dd\Delta \dot f_0\)=0.
$$
By Equation~(\ref{e8}) we have $\(\Dd\dot{f}_1, \Delta \dot f_0\)+\(\dot{f}_1, \Dd\Delta \dot f_0\)=0$
and hence
$$
\(\Dd\ddot{f}_1, \Delta f_0\)+\(\ddot{f}_1, \Dd\Delta f_0\)=0.
$$
We have arrived at Equation~(\ref{e6}).

Finally Equations~(\ref{e7}) is proved by analogy with Equations~(\ref{e6}).
\end{proof}

\subsection{Infinitesimal flexibility of 2-ribbon surfaces}\label{inf_flex_subs}
Our main goal for this subsection is to prove the following general theorem
\begin{theorem}\label{1-d-inf}
Let $f\in C^{1,2,1}_0([a,b],\r^3)$ be a weakly generic 2-ribbon surface with fixed initial position.
Then $f$ has one degree of infinitesimal flexibility.
\end{theorem}

First we write down and investigate a supplementary system of differential
equations (System~A) which describes infinitesimal flexions of
weakly generic 2-ribbon surfaces.
We also show the uniqueness of the solution of
System~A for a given initial data (Proposition~\ref{111}).
The remaining part of this subsection is dedicated to the proof of
Theorem~\ref{1-d-inf} mentioned above.
In Proposition~\ref{condition} we show
that every infinitesimal flexion satisfies System~A. Then in
Proposition~\ref{existence} we prove that every solution of System~A
with certain initial data is an infinitesimal flexion.  After that we prove
Theorem~\ref{1-d-inf}.

\subsubsection{System A}
Let
\begin{equation}\label{g_i}
\begin{array}{lll}
G_{11}=\(\Dd\dot{ f_1}, \dot f_1\), \quad &
G_{12}=\(\Dd\dot{ f_1}, \Delta f_0\), \quad &
G_{13}=\(\Dd\dot{ f_1}, \Delta f_1\), \\
G_{21}=\(\Dd \Delta f_0, \dot f_1\), \quad &
G_{22}=\(\Dd \Delta f_0,  \Delta f_0\), \quad &
G_{23}=\(\Dd \Delta f_0, \Delta f_1\),\\
G_{31}=\(\Dd \Delta f_1, \dot f_1\), \quad &
G_{32}=\(\Dd \Delta f_1,  \Delta f_0\), \quad &
G_{33}=\(\Dd \Delta f_1, \Delta f_1\).
\end{array}
\end{equation}

Denote by {\it System A} the following system of differential
equations
\begin{equation*}
\left\{
\begin{array}{lll}
\dot G_{11}&=&0,\\
\dot G_{12}&=&\left(
\frac{(\dot f_1,\Delta  \dot f_0,\Delta  f_1)}{(\dot f_1,\Delta
f_0,\Delta  f_1)} +
\frac{(\ddot f_1,\Delta  f_0,\Delta f_1)}{(\dot f_1,\Delta
f_0,\Delta  f_1)}
\right)G_{12}+
\frac{(\dot f_1,\Delta  f_0, \Delta \dot f_0)}{(\dot f_1,\Delta
f_0, \Delta  f_1)}G_{13}-
\frac{(\dot f_1,\Delta  f_0, \ddot f_1)}{(\dot f_1,\Delta  f_0, \Delta  f_1)}G_{23},\\
\dot G_{13}&=&\frac{(\dot f_1,\Delta \dot f_1,\Delta  f_1)}{(\dot
f_1,\Delta f_0,\Delta f_1)}G_{12}+
\left(
\frac{(\dot f_1,\Delta  f_0,\Delta  \dot f_1)}{(\dot f_1,\Delta
f_0,\Delta  f_1)}+
\frac{(\ddot f_1,\Delta  f_0,\Delta  f_1)}{(\dot f_1,\Delta
f_0,\Delta  f_1)}
\right) G_{13}-
\frac{(\dot f_1,\ddot f_1,\Delta  f_1)}{(\dot f_1,\Delta
f_0,\Delta  f_1)}G_{32},\\
\dot G_{21}&=&-\left(
\frac{(\dot f_1,\Delta  \dot f_0,\Delta  f_1)}{(\dot f_1,\Delta
f_0,\Delta  f_1)} +
\frac{(\ddot f_1,\Delta  f_0,\Delta f_1)}{(\dot f_1,\Delta
f_0,\Delta  f_1)}
\right)G_{12}-
\frac{(\dot f_1,\Delta  f_0, \Delta \dot f_0)}{(\dot f_1,\Delta
f_0, \Delta  f_1)}G_{13}+
\frac{(\dot f_1,\Delta  f_0, \ddot f_1)}{(\dot f_1,\Delta  f_0, \Delta  f_1)}G_{23},\\
\dot G_{22}&=&0,\\
\dot G_{23}&=&
-\left(
\frac{(\Delta f_1,\Delta f_0,\dot f_1{\times} \Delta f_0)(\dot
f_1,\Delta \dot f_0, \Delta f_1)}{|\dot f_1{\times} \Delta
f_0|^2(\dot f_1,\Delta f_0, \Delta f_1)}
-\frac{(\dot f_1,\Delta f_1,\dot f_1{\times} \Delta f_0)(\Delta
\dot f_0,\Delta f_0,\Delta f_1)}{|\dot f_1{\times} \Delta
f_0|^2(\dot f_1,\Delta f_0,\Delta f_1)}+\right.
\\&&
\left.
\frac{(\dot f_1,\Delta f_0{\times} \Delta \dot f_0,\Delta
f_1)}{|\dot f_1{\times} \Delta f_0|^2}+
\frac{(\dot f_1{\times} \Delta \dot f_0,\Delta f_0,\Delta
f_1)}{|\dot f_1{\times} \Delta f_0|^2}+ \frac{(\Delta \dot
f_1,\Delta f_0,\Delta f_1)}{(\dot f_1,\Delta f_0,\Delta f_1)}
\right) G_{12}-
\\&&
\left(
\frac{(\Delta f_1,\Delta f_0,\dot f_1{\times} \Delta f_0)(\dot
f_1,\Delta  f_0, \Delta \dot f_0)}{|\dot f_1{\times} \Delta
f_0|^2(\dot f_1,\Delta f_0, \Delta f_1)}
+ \frac{(\dot f_1,\Delta f_0, \Delta f_0{\times} \Delta \dot
f_0)}{|\dot f_1{\times} \Delta f_0|^2}
\right)G_{13}-\\
&&
\left(
\frac{(\dot f_1,\Delta f_1,\dot f_1{\times} \Delta f_0)(\dot
f_1,\Delta f_0, \Delta \dot f_0)}{|\dot f_1{\times} \Delta
f_0|^2(\dot f_1,\Delta f_0,\Delta f_1)}
- \frac{(\dot f_1,\Delta f_0, \dot f_1{\times} \Delta \dot
f_0)}{|\dot f_1{\times} \Delta f_0|^2}
-\frac{(\dot f_1,\Delta f_0, \Delta \dot f_1)}{(\dot f_1,\Delta
f_0, \Delta f_1)}\right)G_{23},
\\
\dot G_{31}&=&-\frac{(\dot f_1,\Delta \dot f_1,\Delta  f_1)}{(\dot
f_1,\Delta f_0,\Delta f_1)}G_{12}-
\left(
\frac{(\dot f_1,\Delta  f_0,\Delta  \dot f_1)}{(\dot f_1,\Delta
f_0,\Delta  f_1)}+
\frac{(\ddot f_1,\Delta  f_0,\Delta  f_1)}{(\dot f_1,\Delta
f_0,\Delta  f_1)}
\right) G_{13}+
\frac{(\dot f_1,\ddot f_1,\Delta  f_1)}{(\dot f_1,\Delta
f_0,\Delta  f_1)}G_{32},\\
\dot G_{32}&=&
-\left(
\frac{(\Delta f_0,\Delta f_1,\dot f_1{\times} \Delta f_1)(\dot
f_1,\Delta  \dot f_1, \Delta f_1)}{|\dot f_1{\times} \Delta
f_1|^2(\dot f_1,\Delta f_0, \Delta f_1)}
+\frac{(\dot f_1,\Delta f_1, \Delta f_1{\times} \Delta \dot
f_1)}{|\dot f_1{\times} \Delta f_1|^2}
\right)G_{12}-\\
&& \left(
\frac{(\Delta f_0,\Delta f_1,\dot f_1{\times} \Delta f_1)(\dot
f_1,\Delta f_0, \Delta \dot f_1)}{|\dot f_1{\times} \Delta
f_1|^2(\dot f_1,\Delta f_0, \Delta f_1)}
-\frac{(\dot f_1,\Delta f_0,\dot f_1{\times} \Delta f_1)(\Delta
\dot f_1,\Delta f_0,\Delta f_1)}{|\dot f_1{\times} \Delta
f_1|^2(\dot f_1,\Delta f_0,\Delta f_1)}+\right.
\\&&
\left.
\frac{(\dot f_1,\Delta f_1{\times} \Delta \dot f_1,\Delta
f_0)}{|\dot f_1{\times} \Delta f_1|^2}+
\frac{(\dot f_1{\times} \Delta \dot f_1,\Delta f_1,\Delta
f_0)}{|\dot f_1{\times} \Delta f_1|^2}+
\frac{(\Delta \dot f_0,\Delta f_0,\Delta f_1)}{(\dot f_1,\Delta
f_0,\Delta f_1)}
\right) G_{13}-
\\
&&
\left(
\frac{(\dot f_1,\Delta f_0,\dot f_1{\times} \Delta f_1)(\dot
f_1,\Delta \dot f_1, \Delta f_1)}{|\dot f_1{\times} \Delta
f_1|^2(\dot f_1,\Delta f_0,\Delta f_1)}
- \frac{(\dot f_1,\Delta f_1, \dot f_1{\times} \Delta \dot
f_1)}{|\dot f_1{\times} \Delta f_1|^2}
-\frac{(\dot f_1,\Delta \dot f_0, \Delta f_1)}{(\dot f_1,\Delta
f_0, \Delta f_1)}\right)G_{32},
\\
\dot G_{33}&=&0.\\
\end{array}
\right.
\end{equation*}

\begin{remark}
In Proposition~\ref{inner} below we show an explicit formula for
the function $G_{23}{+}G_{32}$, it is $\Phi$ in our notation of
Section~2.
\end{remark}

Note also that $\dot G_{12}+ \dot G_{21}=0$ and $\dot G_{13} + \dot G_{31}=0$
in System A.

\begin{example}
Let us consider a simple example of a 2-ribbon curve where
$\dot f$, $\Delta f_0$, and $\Delta f_1$ are all constants.
Let us call these surfaces {\it book-shaped surfaces}. Direct calculations show that
$$
\dot G_{11}=\dot G_{12}=\ldots =\dot G_{33}=0
$$
(this happens, since all the summands in the coefficients of System~A contain
either $\ddot f_1$, or $\Delta \dot f_0$, or $\Delta \dot f_1$ which are all zeroes in our case).
Hence all the scalar products of the deformation with vectors $\dot f_1, \Delta f_0, \Delta f_1$ do not depend on $t$.
Therefore, every element of every isometric deformations of a book-shaped surface is a book-shaped surface.
Here is a typical example of isometric deformation in this class:
$$
f_1^\lambda(t)=(t,0,0), \quad \Delta f_0^\lambda(t)=(0,1,0), \quad \Delta_1^{\lambda}(t)=(0,\sin \lambda,\cos\lambda).
$$
This deformation can be geometrically seen as an opening a museum book with two rigid plastic pages.
\end{example}

In the following proposition we prove that for every single 2-ribbon surface $f$ (not for a
deformation) and initial data for $G_{ij}$ at one point $f(t_0)$
System~A has a unique solution.
Recall that $t$ is an argument of $f$.

\begin{proposition}\label{111}
Let $f$ be a weakly generic 2-ribbon surface in $C^{1,2,1}([a,b],\r^3)$.
For every collection of initial data $G_{ij}(a)=c_{ij}$ there exists a
unique solution of System~A on $[a,b]$.
\end{proposition}

\begin{proof}
System A is the system of homogeneous linear differential equations with smooth variable coefficients (since
$(\dot f_1,\Delta f_0,\Delta f_1)$ never vanishes on $[a,b]$)
and hence for every collection of initial data it has a unique
solution on the segment $[a,b]$.
\end{proof}

\subsubsection{Every infinitesimal flexion satisfies System~A}

Let us show the following statement.

\begin{proposition}\label{condition}
Let $f$ be a weakly generic 2-ribbon surface in $C^{1,2,1}([a,b],\r^3)$.
Then for every infinitesimal flexion $\Dd f$ the functions
$G_{11},G_{12}, \ldots , G_{33}$ satisfy system~A.
\end{proposition}

We start the proof with the following general lemma.

\begin{lemma}\label{additional}
For every infinitesimal flexion $\Dd f$ we have the equalities
$$
G_{11}=G_{22}=G_{33}=0, \qquad G_{12}+G_{21}=0, \quad  \hbox{and} \quad G_{13} +
G_{31}=0.
$$
\end{lemma}

\begin{proof}
The functions $|\dot f_1|$, $|\Delta f_0|$, and $|\Delta f_1|$ are
infinitesimally preserved by infinitesimal flexions, hence $G_{11}$,
$G_{22}$, and $G_{33}$ vanish.

The invariance of angles between $\dot f_1$ and $\Delta f_0$, and
$\dot f_1$ and $\Delta f_1$ yield the equations $G_{12}+G_{21}=0$ and
$G_{13}+G_{31}=0$, respectively.
\end{proof}

{\it Proof of Proposition~\ref{condition}.} From
Lemma~\ref{additional} the functions $G_{11}$, $G_{22}$, and $G_{33}$ are
zero functions, thus $\dot G_{11}$, $\dot G_{22}$, and $\dot G_{33}$
are zero functions as well.

\vspace{2mm}

Let us prove the expression for $\dot G_{12}$ and $\dot G_{13}$. Note
that
$$
\dot G_{12}=\(\Dd \ddot f_1, \Delta f_0\)+\(\Dd \dot f_1,\Delta \dot
f_0\).
$$
Thus Equations~(\ref{e8}) and~({\ref{e6}}) imply
$$
\dot G_{12}=\(\Dd \dot f_1, \Delta\dot f_0\)-\(\ddot f_1, \Dd \Delta
f_0\).
$$
To obtain the expression for $\dot G_{12}$ rewrite $\Delta \dot f_0$
and $\ddot f_1$ in the basis consisting of vectors $\dot f_1$,
$\Delta f_0$, and $\Delta f_1$.
\begin{align*}
\dot G_{12}&=\(\Dd \dot f_1, \Delta\dot f_0\)-\(\ddot f_1, \Dd \Delta
f_0\)
\\
&=
\left(
\frac{(\Delta  \dot f_0,\Delta f_0,\Delta  f_1)}{(\dot f_1,\Delta f_0,\Delta  f_1)}G_{11} +
\frac{(\dot f_1,\Delta  \dot f_0,\Delta  f_1)}{(\dot f_1,\Delta f_0,\Delta  f_1)}G_{12} +
\frac{(\dot f_1,\Delta  f_0, \Delta \dot f_0)}{(\dot f_1,\Delta f_0, \Delta  f_1)}G_{13}
\right)
\\
&\quad-\left(
\frac{(\ddot f_1,\Delta  f_0,\Delta f_1)}{(\dot f_1,\Delta f_0,\Delta  f_1)}G_{21}+
\frac{(\dot f_1,\ddot f_1,\Delta f_1)}{(\dot f_1,\Delta f_0,\Delta  f_1)}G_{22}
+\frac{(\dot f_1,\Delta  f_0, \ddot f_1)}{(\dot f_1,\Delta  f_0, \Delta  f_1)}G_{23}
\right)
\\
&=
\left(
\frac{(\dot f_1,\Delta  \dot f_0,\Delta  f_1)}{(\dot f_1,\Delta f_0,\Delta  f_1)} +
\frac{(\ddot f_1,\Delta  f_0,\Delta f_1)}{(\dot f_1,\Delta f_0,\Delta  f_1)}
\right)G_{12}+
\frac{(\dot f_1,\Delta  f_0, \Delta \dot f_0)}{(\dot f_1,\Delta f_0, \Delta  f_1)}G_{13}-
\frac{(\dot f_1,\Delta  f_0, \ddot f_1)}{(\dot f_1,\Delta  f_0, \Delta  f_1)}G_{23}.
\end{align*}
The last equation holds since $G_{11}=0$, $G_{22}=0$, and $G_{21}=-G_{12}$.

The same strategy works for the functions $\dot G_{13}$.

\vspace{2mm}

Now we study expressions for $\dot G_{21}$ and $\dot G_{31}$. From
Lemma~\ref{additional} we know that $G_{21}=-G_{12}$ and $G_{31}=-G_{13}$ and
hence $\dot G_{21}=-\dot G_{12}$ and $\dot G_{31}=-\dot G_{13}$. Therefore,
the equations for $\dot G_{21}$ and $\dot G_{31}$ are satisfied.

\vspace{2mm}

In order to get the expression for $\dot G_{23}$, we first show
that the function $(\dot f_1, \Delta f_0, \Delta \dot f_0)$ is an invariant of
infinitesimal flexions. Indeed,
$$
(\dot f_1, \Delta f_0, \Delta \dot f_0)=
(\dot f_1, \Delta f_0, \dot f_1{-}\dot f_0)=
-(\dot f_1, \Delta f_0, \dot f_0).
$$
The vectors $\dot f_0$, $\dot f_1$, and $\Delta f_0$ form a rigid frame, hence their triple product
is an invariant  of infinitesimal flexions. Hence the function $(\dot f_1, \Delta f_0, \Delta \dot f_0)$
is an invariant as well.

The infinitesimal flexion invariance of $(\dot f_1, \Delta f_0, \Delta \dot f_0)$  implies that
$\Dd (\dot f_1, \Delta f_0, \Delta \dot f_0)=0$. So we get
$$
(\Dd\dot{ f}_1, \Delta f_0, \Delta \dot f_0)+ (\dot f_1, \Dd \Delta
f_0, \Delta \dot f_0)+ (\dot f_1, \Delta f_0, \Dd \Delta \dot
f_0)=0.
$$
Rewrite
\begin{align*}
(\dot f_1,  \Delta f_0, \Dd \Delta \dot f_0)&=-(\Dd\dot f_1, \Delta
f_0, \Delta \dot f_0)- (\dot f_1, \Dd \Delta f_0, \Delta \dot f_0)
\\&=
-\(\Dd\dot f_1, \Delta f_0{\times} \Delta \dot f_0\)+ \(\Dd \Delta f_0,\dot f_1{\times} \Delta \dot f_0\)\\
&
\begin{array}{l}
=-\frac{(\Delta f_0{\times} \Delta \dot f_0,\Delta f_0,\Delta
f_1)}{(\dot f_1,\Delta f_0,\Delta f_1)}G_{11}-
\frac{(\dot f_1,\Delta f_0{\times} \Delta \dot f_0,\Delta
f_1)}{(\dot f_1,\Delta f_0,\Delta f_1)}G_{12}-
\frac{(\dot f_1,\Delta f_0, \Delta f_0{\times} \Delta \dot
f_0)}{(\dot f_1,\Delta f_0, \Delta f_1)}G_{13}+\\
\quad \frac{(\dot f_1{\times} \Delta \dot f_0,\Delta f_0,\Delta
f_1)}{(\dot f_1,\Delta f_0,\Delta f_1)}G_{21}+
\frac{(\dot f_1,\dot f_1{\times} \Delta \dot f_0,\Delta
f_1)}{(\dot f_1,\Delta f_0,\Delta f_1)}G_{22}+
\frac{(\dot f_1,\Delta f_0, \dot f_1{\times} \Delta \dot
f_0)}{(\dot f_1,\Delta f_0, \Delta f_1)}G_{23}.
\end{array}\\
\end{align*}
Second, we have
$$
\begin{array}{l}
\(\Dd \Delta \dot f_0,\Delta f_0\)=-\({\Dd \Delta f_0},\Delta \dot
f_0\)=
-\frac{(\Delta \dot f_0,\Delta f_0,\Delta f_1)}{(\dot f_1,\Delta
f_0,\Delta f_1)}G_{21}-
\frac{(\dot f_1,\Delta \dot f_0,\Delta f_1)}{(\dot f_1,\Delta
f_0,\Delta f_1)}G_{22}-
\frac{(\dot f_1,\Delta f_0, \Delta \dot f_0)}{(\dot f_1,\Delta
f_0, \Delta f_1)}G_{23} .
\end{array}
$$
Third, we get
$$
\(\Dd\Delta\dot f_0,\dot f_1\)=-\(\Dd\dot f_1,\Delta \dot f_0\)=
\begin{array}{l}
-\frac{(\dot f_1,\Delta  \dot f_0, \Delta f_1)}{(\dot f_1,\Delta
f_0, \Delta f_1)}G_{12}-
\frac{(\dot f_1,\Delta  f_0, \Delta \dot f_0)}{(\dot f_1,\Delta
f_0, \Delta f_1)}G_{13}.
\end{array}
$$
Fourth,
\begin{align*}
\(\Dd \Delta \dot f_0,\Delta f_1\)=&
\begin{array}{l} \frac{(\Delta f_1,\Delta f_0,\dot f_1{\times} \Delta
f_0)}{(\dot f_1,\Delta f_0,\dot f_1{\times} \Delta f_0)}
\(\Dd\Delta \dot f_0,\dot f_1\)+
\frac{(\dot f_1,\Delta f_1,\dot f_1{\times} \Delta f_0)}{(\dot
f_1,\Delta f_0,\dot f_1{\times} \Delta f_0)}\(\Dd\Delta \dot
f_0,\Delta f_0\)+
\end{array}
\\&
\begin{array}{l}
\frac{(\dot f_1,\Delta f_0,\Delta f_1)}{(\dot f_1,\Delta f_0,\dot
f_1{\times} \Delta f_0)}(\dot f_1, \Delta f_0,\Dd\Delta \dot f_0).
\end{array}\\
\end{align*}
After the substitution of the four above expressions and
simplifications we have
\begin{align*}
\(\Dd\Delta \dot f_0,\Delta f_1\)= &\begin{array}{l}
-\left(
\frac{(\Delta f_1,\Delta f_0,\dot f_1{\times} \Delta f_0)(\dot
f_1,\Delta \dot f_0, \Delta f_1)}{|\dot f_1{\times} \Delta
f_0|^2(\dot f_1,\Delta f_0, \Delta f_1)}
-\frac{(\dot f_1,\Delta f_1,\dot f_1{\times} \Delta f_0)(\Delta
\dot f_0,\Delta f_0,\Delta f_1)}{|\dot f_1{\times} \Delta
f_0|^2(\dot f_1,\Delta f_0,\Delta f_1)}+\right.
\end{array}
\\&
\begin{array}{l}
\left.
\frac{(\dot f_1,\Delta f_0{\times} \Delta \dot f_0,\Delta
f_1)}{|\dot f_1{\times} \Delta f_0|^2}+
\frac{(\dot f_1{\times} \Delta \dot f_0,\Delta f_0,\Delta
f_1)}{|\dot f_1{\times} \Delta f_0|^2}
\right) G_{12}-
\end{array}
\\&
\begin{array}{l}
\left(
\frac{(\Delta f_1,\Delta f_0,\dot f_1{\times} \Delta f_0)(\dot
f_1,\Delta  f_0, \Delta \dot f_0)}{|\dot f_1{\times} \Delta
f_0|^2(\dot f_1,\Delta f_0, \Delta f_1)}
+ \frac{(\dot f_1,\Delta f_0, \Delta f_0{\times} \Delta \dot
f_0)}{|\dot f_1{\times} \Delta f_0|^2}
\right)G_{13}-
\end{array}\\
&\begin{array}{l}
\left(
\frac{(\dot f_1,\Delta f_1,\dot f_1{\times} \Delta f_0)(\dot
f_1,\Delta f_0, \Delta \dot f_0)}{|\dot f_1{\times} \Delta
f_0|^2(\dot f_1,\Delta f_0,\Delta f_1)}
-
\frac{(\dot f_1,\Delta f_0, \dot f_1{\times} \Delta \dot
f_0)}{|\dot f_1{\times} \Delta f_0|^2}
\right)G_{23}.
\end{array}
\end{align*}
Further, decomposing the vector $\Delta \dot f_1$ into basis vectors $\dot f_1$, $\Delta
f_0$, and $\Delta f_1$ we get
$$
\begin{array}{l}
\(\Dd \Delta f_0, \Delta \dot f_1\)=
\frac{(\Delta \dot f_1,\Delta f_0,\Delta f_1)}{(\dot f_1,\Delta
f_0,\Delta f_1)}G_{21}+
\frac{(\dot f_1,\Delta \dot f_1,\Delta f_1)}{(\dot f_1,\Delta
f_0,\Delta f_1)}G_{22}+
\frac{(\dot f_1,\Delta f_0, \Delta \dot f_1)}{(\dot f_1,\Delta
f_0, \Delta f_1)}G_{23}.
\end{array}
$$
From the last two identities, by substituting  $G_{22}=0$ and
$G_{21}=-G_{12}$ (see Lemma~\ref{additional}), we obtain the expression
for
$$
\dot G_{23}=\frac{\partial}{\partial t}\(\Dd \Delta f_0, \Delta f_1\)=
\(\Dd \Delta \dot f_0, \Delta f_1\)+\(\Dd \Delta f_0, \Delta \dot f_1\).
$$

The expression for $\dot G_{32}$ is calculated in a similar way. This
concludes the proof. \qed

\vspace{2mm}

\subsubsection{Existence of infinitesimal flexions}

Let us prove that every solution of System~A with certain
initial data determines an infinitesimal flexion.

\begin{proposition}\label{existence}
Let $f$ be a weakly generic 2-ribbon surface in $C^{1,2,1}([a,b],\r^3)$.
Then

{\noindent
$($i$)$ For an arbitrary nonzero $\alpha$
there exists a unique tangent vector $\Dd f$ at $f$ satisfying System~A and the boundary conditions
$$
\Dd\dot{ f_1}(a)=0,  \quad \Dd \Delta f_0(a)=0, \quad  \hbox{and}
\quad \Dd \Delta f_1(a)=\alpha \dot f_1(a){\times} \Delta f_1(a).
$$
}

{
\noindent
$($ii$)$ This tangent vector is an infinitesimal flexion.
}
\end{proposition}

\begin{remark}
Here and below, for a function $f$ defined on $[a,b]$ by $\dot f(a)$ we mean the one-sided derivative at $a$.
\end{remark}

\begin{proof} We start with Proposition~\ref{existence}$(i)$.
Consider three vectors
$$
v_1=0,  \quad v_2=0, \quad
\hbox{and} \quad v_3=\alpha \dot f_1(a){\times} \Delta f_1(a).
$$
Denote
$$
\begin{array}{lll}
c_{11}=\(v_1, \dot f_1\), \quad &
c_{12}=\(v_1, \Delta f_0\), \quad &
c_{13}=\(v_1, \Delta f_1\), \\
c_{21}=\(v_2, \dot f_1\), \quad &
c_{22}=\(v_2,  \Delta f_0\), \quad &
c_{23}=\(v_2, \Delta f_1\),\\
c_{31}=\(v_3, \dot f_1\), \quad &
c_{32}=\(v_3,  \Delta f_0\), \quad &
c_{33}=\(v_3, \Delta f_1\).
\end{array}
$$
By Proposition~\ref{111} there exists a unique solution $(G_{11}, G_{12},\ldots, G_{33})$ satisfying the initial conditions $G_{ij}(a)=c_{ij}$.
For every point $t\in [a,b]$ the values $\Dd \dot f_1$, $\Dd \Delta f_0$, and $\Dd \Delta f_1$ of the tangent vector $\Dd f$
are uniquely defined in the basis $(\dot f_1, \Delta f_0, \Delta f_1)$ by Equations~(\ref{g_i}):
here we substitute the solution of System~A with the initial conditions $G_{ij}(a)=c_{ij}$
to the right hand side of Equations~(\ref{g_i}).
Hence, there exists a unique tangent vector $\Dd f$ of $f$ satisfying System~A and the boundary conditions
$$
\Dd\dot{ f_1}(a)=0,  \quad \Dd \Delta f_1(a)=0, \quad  \hbox{and}
\quad \Dd \Delta f_0(a)=\alpha \dot f_1(a){\times} \Delta f_0(a).
$$
This concludes the proof of the fist item of the proposition.

\vspace{2mm}

{
\noindent
{\it Proof of Proposition~\ref{existence}$($ii$)$.}
By the definition of an infinitesimal flexion
it is enough to check that the following 11 functions are
preserved by the infinitesimal deformation:
$$
|\dot f_i|, \qquad  |\Delta f_{i}|,  \qquad \(\dot f_i,\Delta
f_{i-1}\), \qquad\(\dot f_i,\Delta f_{i}\), \quad \hbox{and} \quad
\(\dot f_i,\dot f_{i+1}\)
$$
(for all possible admissible $i$).
}

\vspace{2mm}

{\it Invariance of $|\dot f_1|$, $|\Delta f_0|$, $|\Delta f_1|$,
$\(\dot f_1,\Delta f_0\)$, and $\(\dot f_1,\Delta f_1\)$}.

From System~A we have
$$
\dot G_{11}=0, \quad \dot G_{22} =0, \quad  \dot G_{33} =0, \quad \dot G_{21}
+\dot G_{12} =0 , \quad \dot G_{31} +\dot G_{13} =0,
$$
and hence the functions
$$
\begin{array}{c}
\Dd(|\dot f_1|^2)=2G_{11}; \quad \Dd(|\Delta f_0|^2)=2G_{22}; \quad \Dd(|\Delta f_1|^2)=2G_{33};\\
\Dd\(\dot f_1, \Delta f_0 \)=G_{12}+G_{21}, \quad \hbox{and} \quad
\Dd\(\dot f_1, \Delta f_1 \)=G_{31}+G_{13}
\end{array}
$$
are constant functions. So
it is enough to show that they vanish at some point: we show this
at point $a$.
$$
\begin{array}{l}
\begin{array}{l}
\Dd\(\dot f_1(a), \dot f_1 (a)\)=2\(\Dd \dot f_1(a), \dot f_1
(a)\)=2\(0, \dot f_1(a)\)=0;
\end{array}
\\
\begin{array}{l}
\Dd\(\Delta f_0(a), \Delta f_0 (a)\)=
2\(\Dd\Delta f_0(a), \Delta f_0 (a)\)=
2(0,\Delta f_0 (a)\)=0;
\end{array}
\\
\begin{array}{l}
\Dd\(\Delta f_1(a), \Delta f_1 (a)\)=
2\(\Dd\Delta f_1(a), \Delta f_1 (a)\)=
2\(\alpha \dot f_1(a){\times} \Delta f_1(a), \Delta f_1 (a)\)=0;
\end{array}
\\
\begin{array}{l}
\Dd\(\dot f_1(a), \Delta f_0 (a)\)=\(\D\dot f_1(a), \Delta f_0
(a)\)+ \(\dot f_1(a), \Dd\Delta f_0 (a)\)=\(0, \Delta f_0 (a)\)+
\\
\qquad \qquad \qquad \qquad \quad \(\dot f_1(a), 0\)=0.
\end{array}
\\
\begin{array}{l}
\Dd\(\dot f_1(a), \Delta f_1 (a)\)=\(\Dd\dot f_1(a), \Delta f_1
(a)\)+ \(\dot f_1(a), \Dd\Delta f_1 (a)\)=\(0, \Delta f_0 (a)\)+
\\
\qquad \qquad \qquad \qquad \quad \(\dot f_1(a), \alpha \dot
f_1(a){\times} \Delta f_1(a)\)=0;
\end{array}
\end{array}
$$

\vspace{2mm}

{\it Invariance of $\(\dot f_0, \Delta f_0 \)$ and $\(\dot f_2,
\Delta f_1 \)$}. Note that
$$
\(\dot f_0, \Delta f_0 \)=-\frac{1}{2} \frac{\partial}{\partial t}
\(\Delta f_0, \Delta f_0\)+ \(\dot f_1,\Delta f_0\).
$$
Hence by the above item we have
$$
\Dd\(\dot f_0,\Delta f_0\)=-\frac{1}{2} \frac{\partial}{\partial t}
\Dd\(\Delta f_0, \Delta f_0\)+ \Dd\(\dot f_1,\Delta f_0\)=-\frac{1}{2} \frac{\partial}{\partial t}(0)+0=0.
$$
Similar reasoning shows that $\Dd\(\dot f_2,\Delta f_1\)=0$.

\vspace{2mm}

{\it Invariance of $\(\dot f_0,\dot f_1\)$ and $\(\dot f_1,\dot
f_2\)$}. Let us prove that $\Dd\(\dot f_0,\dot f_1\)=0$. First,
note that
$$
\(\Dd \dot f_0, \dot f_1\)=\(\Dd\dot f_1, \dot f_1\)-\(\Dd \Delta
\dot f_0, \dot f_1\)=-\(\Dd \Delta \dot f_0, \dot f_1\)= \(\Dd\Delta
f_0, \ddot f_1\)-\frac{\partial}{\partial t}\(\Dd\Delta f_0, \dot
f_1\).
$$
Recall that $\frac{\partial}{\partial t}\(\Dd\Delta f_0, \dot
f_1\)=\dot G_{21}=-\dot G_{12}$. Let us substitute the expression for
$\dot G_{12}$ of System~A and rewrite $\ddot f_1$ in the basis of
vectors $\dot f_1$, $\Delta f_0$, and $\Delta f_1$. One obtains

\begin{align*}
\(\Dd \dot f_0, \dot f_1\)&=\(\Dd\Delta f_0, \ddot f_1\)+\dot G_{12}
\\
&=
\frac{(\dot f_1,\Delta  \dot f_0,\Delta  f_1)}{(\dot f_1,\Delta
f_0,\Delta  f_1)}\(\Dd\dot f_1,\Delta f_0\)+
\frac{(\dot f_1,\Delta  f_0, \Delta \dot f_0)}{(\dot f_1,\Delta
f_0, \Delta  f_1)}\(\Dd\dot f_1,\Delta f_1\)
\\
&=\(\Dd \dot f_1,\Delta
\dot f_0\)=-\(\Dd \dot f_1,\dot f_0\).\\
\end{align*}
Hence
$$
\Dd\(\dot f_0,\dot f_1\)=
\(\Dd \dot f_0, \dot f_1\)+\(\Dd \dot f_1,\dot f_0\)=
-\(\Dd \dot f_1,\dot f_0\)+\(\Dd \dot f_1,\dot f_0\)
=0.
$$
Therefore, $\(\dot f_0,\dot f_1\)$ is invariant under the
infinitesimal deformation. The proof of the invariance of $\(\dot
f_1,\dot f_2\)$ is analogous.

\vspace{2mm}

{\it Invariance of $\(\dot f_0,\dot f_0\)$ and $\(\dot f_2,\dot
f_2\)$}. Let us prove that $\Dd\(\dot f_0,\dot f_0\)=0$.

$$
\Dd\(\dot f_0,\dot f_0\)=2\(\Dd\dot f_0,\dot f_0\)=
2\(\Dd\Delta \dot f_0,\Delta \dot f_0\)+2\Dd\(\dot f_1, \dot f_0\)-
2\(\Dd \dot f_1,\dot f_1\).
$$
We have already shown that $\Dd\(\dot f_1, \dot f_0\)=0$ and $\(\Dd
\dot f_1,\dot f_1\)=0$. Hence
$$
\Dd\(\dot f_0,\dot f_0\)=2\(\Dd\Delta\dot f_0,\Delta\dot f_0\).
$$
We rewrite the last $\Delta\dot f_0$ in the last expression in the
basis $\dot f_1, \Delta f_0, \dot f_1 {\times}\Delta f_0$ and get

\begin{align}\label{eqS}
\begin{array}{l}
(\Dd\Delta\dot f_0,\Delta\dot f_0\)=
\begin{array}{l}
\frac{(\Delta \dot f_0,\Delta  f_0, \dot f_1 {\times}\Delta
f_0)}{(\dot f_1,\Delta  f_0, \dot f_1 {\times}\Delta
f_0)}\(\Dd\Delta\dot f_0,\dot f_1\)+
\frac{(\dot f_1,\Delta\dot  f_0, \dot f_1 {\times}\Delta
f_0)}{(\dot f_1,\Delta  f_0, \dot f_1 {\times}\Delta
f_0)}\(\Dd\Delta\dot f_0,\Delta f_0\)+
\end{array}\\
\qquad\qquad\qquad\quad
\begin{array}{l}
\frac{(\dot f_1,\Delta  f_0, \Delta\dot f_0)}{(\dot f_1,\Delta
f_0, \dot f_1 {\times}\Delta f_0)}(\Dd\Delta\dot f_0,\dot
f_1,\Delta f_0).
\end{array}
\end{array}
\end{align}

Let us rewrite $\(\Dd\Delta\dot f_0,\dot f_1\)$, $\(\Dd\Delta\dot
f_0,\Delta f_0\)$, and $(\Dd\Delta\dot f_0,\dot f_1,\Delta f_0)$ in
terms of $G_{11}, \ldots, G_{33}$. First, we have:
$$
\(\Dd\Delta\dot f_0,\dot f_1\)=\(\Dd\dot f_0,\dot f_1\)=-\(\Dd\dot
f_1,\dot f_0\)=-\(\Dd\dot f_1,\Delta \dot f_0\).
$$
The second equality holds since we have shown that $\Dd\(\dot f_0,
\dot f_1\)=0$. If we rewrite $\Delta \dot f_0$ in the basis $\dot
f_1, \Delta f_0, \Delta f_1$, we get the following:
$$
\begin{array}{l}
\(\Dd\Delta \dot f_0,\dot f_1\)=-\(\Dd\dot f_1,\Delta \dot f_0\)=-\frac{(\dot f_1,\Delta  \dot f_0, \Delta
f_1)}{(\dot f_1,\Delta f_0, \Delta f_1)}G_{12}-
\frac{(\dot f_1,\Delta  f_0, \Delta \dot f_0)}{(\dot f_1,\Delta
f_0, \Delta f_1)}G_{13}.
\end{array}
$$
Second, we have
$$
\begin{array}{l}
\(\Dd \Delta \dot f_0,\Delta f_0\)=-\({\Dd \Delta f_0},\Delta \dot
f_0\)=
\frac{(\Delta \dot f_0,\Delta f_0,\Delta f_1)}{(\dot f_1,\Delta
f_0,\Delta f_1)}G_{12}-
\frac{(\dot f_1,\Delta f_0, \Delta \dot f_0)}{(\dot f_1,\Delta
f_0, \Delta f_1)}G_{23}.
\end{array}
$$
Third, with
\begin{align*}
\dot G_{23}-\(\Dd \Delta f_0,\Delta \dot f_1\) =&\(\Dd \Delta \dot
f_0,\Delta f_1\)=
\begin{array}{l} \frac{(\Delta f_1,\Delta f_0,\dot f_1{\times} \Delta
f_0)}{(\dot f_1,\Delta f_0,\dot f_1{\times} \Delta f_0)}
\(\Dd\Delta \dot f_0,\dot f_1\)+
\end{array}
\\&
\begin{array}{l}
\frac{(\dot f_1,\Delta f_1,\dot f_1{\times} \Delta f_0)}{(\dot
f_1,\Delta f_0,\dot f_1{\times} \Delta f_0)}\(\Dd\Delta \dot
f_0,\Delta f_0\)+
\frac{(\dot f_1,\Delta f_0,\Delta f_1)}{(\dot f_1,\Delta f_0,\dot
f_1{\times} \Delta f_0)}(\Dd\Delta \dot f_0,\dot f_1, \Delta f_0).
\end{array}\\
\end{align*}
and the expression for $\dot G_{23}$ of System~A we get:
\begin{align*}
(\Dd \Delta \dot f_0, \dot f_1,  \Delta f_0)= &
\begin{array}{l}
-\left(\frac{(\dot f_1{\times} \Delta \dot f_0,\Delta f_0,\Delta
f_1)}{(\dot f_1,\Delta f_0,\Delta f_1)}+
\frac{(\dot f_1,\Delta f_0{\times} \Delta \dot f_0,\Delta
f_1)}{(\dot f_1,\Delta f_0,\Delta f_1)}\right)G_{12}-
\frac{(\dot f_1,\Delta f_0, \Delta f_0{\times} \Delta \dot
f_0)}{(\dot f_1,\Delta f_0, \Delta f_1)}G_{13}+
\end{array}
\\
&
\begin{array}{l}
\frac{(\dot f_1,\Delta f_0, \dot f_1{\times} \Delta \dot
f_0)}{(\dot f_1,\Delta f_0, \Delta f_1)}G_{23}.
\end{array}\\
\end{align*}
Finally, we substitute the obtained last three expressions for
$$
\(\Dd\Delta\dot f_0,\dot f_1\), \quad \(\Dd\Delta\dot f_0,\Delta f_0\), \quad
\hbox{and} \quad
(\Dd\Delta\dot f_0,\dot f_1,\Delta f_0)
$$
respectively to Expression~(\ref{eqS}) and arrive at
\begin{align*}
\(\Dd\Delta\dot f_0,\Delta\dot f_0\)=&
\begin{array}{l}
\left(-\frac{(\Delta \dot f_0,\Delta  f_0, \dot f_1 {\times}\Delta
f_0)(\dot f_1,\Delta  \dot f_0, \Delta f_1)}{(\dot f_1,\Delta f_0,
\dot f_1 {\times}\Delta f_0)(\dot f_1,\Delta f_0, \Delta f_1)}
+\frac{(\dot f_1,\Delta\dot  f_0, \dot f_1 {\times}\Delta
f_0)(\Delta \dot f_0,\Delta f_0,\Delta f_1)}{(\dot f_1,\Delta f_0,
\dot f_1 {\times}\Delta f_0)(\dot f_1,\Delta f_0,\Delta f_1)}-
\right.
\end{array}
\\ &
\begin{array}{l}
\left.
\frac{(\dot f_1,\Delta  f_0, \Delta\dot f_0)(\dot f_1{\times}
\Delta \dot f_0,\Delta f_0,\Delta f_1)}{(\dot f_1,\Delta f_0, \dot
f_1 {\times}\Delta f_0)(\dot f_1,\Delta f_0,\Delta f_1)}
-\frac{(\dot f_1,\Delta  f_0, \Delta\dot f_0)(\dot f_1,\Delta
f_0{\times} \Delta \dot f_0,\Delta f_1)}{(\dot f_1,\Delta f_0,
\dot f_1 {\times}\Delta f_0)(\dot f_1,\Delta f_0,\Delta f_1)}
\right) G_{12}+
\end{array}
\\ &
\begin{array}{l}
\left(
-\frac{(\Delta \dot f_0,\Delta  f_0, \dot f_1 {\times}\Delta
f_0)(\dot f_1,\Delta  f_0, \Delta \dot f_0)}{(\dot f_1,\Delta f_0,
\dot f_1 {\times}\Delta f_0)(\dot f_1,\Delta f_0, \Delta f_1)}
-\frac{(\dot f_1,\Delta  f_0, \Delta\dot f_0)(\dot f_1,\Delta f_0,
\Delta f_0{\times} \Delta \dot f_0)}{(\dot f_1,\Delta f_0, \dot
f_1 {\times}\Delta f_0)(\dot f_1,\Delta f_0, \Delta f_1)}
\right)G_{13}+
\end{array}
\\ &
\begin{array}{l}
\left(
-\frac{(\dot f_1,\Delta\dot  f_0, \dot f_1 {\times}\Delta
f_0)(\dot f_1,\Delta f_0, \Delta \dot f_0)}{(\dot f_1,\Delta  f_0,
\dot f_1 {\times}\Delta f_0)(\dot f_1,\Delta f_0, \Delta f_1)}
+\frac{(\dot f_1,\Delta  f_0, \Delta\dot f_0)(\dot f_1,\Delta f_0,
\dot f_1{\times} \Delta \dot f_0)}{(\dot f_1,\Delta f_0, \dot f_1
{\times}\Delta f_0)(\dot f_1,\Delta f_0, \Delta f_1)}
\right)G_{23}.
\end{array}
\end{align*}

It is clear that the coefficients of $G_{13}$ and $G_{23}$ vanish
identically. Let us study the coefficient of $G_{12}$.

Consider the following mixed product $(\Delta \dot f_0,\Delta \dot
f_0, \dot f_1 {\times}\Delta f_0)$, it is identical to zero. Let
us rewrite $\Delta \dot f_0$ in the second position of the mixed
product in the basis $\dot f_0$, $\Delta f_0$, $\Delta f_1$. We
get the relation
$$
\begin{array}{l}
\frac{(\Delta \dot f_0,\Delta  f_0, \Delta  f_1)}{(\dot f_1,\Delta
f_0, \Delta f_1)}(\Delta \dot f_0,\dot f_1, \dot f_1
{\times}\Delta f_0)+
\frac{(\dot f_1,\Delta \dot f_0, \Delta  f_1)}{(\dot f_1,\Delta
f_0, \Delta  f_1)}(\Delta \dot f_0,\Delta f_0, \dot f_1
{\times}\Delta f_0)\\
=-\frac{(\dot f_1,\Delta  f_0, \Delta \dot f_0)}{(\dot f_1,\Delta
f_0, \Delta  f_1)}(\Delta \dot f_0,\Delta f_1, \dot f_1
{\times}\Delta f_0).
\end{array}
$$
We apply this identity to the first two summands of the
coefficient of $G_{12}$ and get the following expression for the
coefficient of $G_{12}$:
$$
\begin{array}{l}
\frac{(\dot f_1,\Delta  f_0, \Delta \dot f_0)(\Delta \dot
f_0,\Delta f_1, \dot f_1 {\times}\Delta f_0)}{(\dot f_1,\Delta
f_0, \Delta f_1)|\dot f_1{\times} \Delta f_0|^2}-
\frac{(\dot f_1,\Delta f_0 {\times}\Delta \dot f_0,\Delta
f_1)(\dot f_1,\Delta f_0, \Delta \dot f_0)}{(\dot f_1,\Delta f_0,
\Delta f_1)|\dot f_1{\times} \Delta f_0|^2}-
\frac{(\dot f_1 {\times} \Delta \dot f_0,\Delta f_0, \Delta
f_1)(\dot f_1,\Delta  f_0, \Delta \dot f_0)}{(\dot f_1,\Delta f_0,
\Delta f_1)|\dot f_1{\times} \Delta f_0|^2}.
\end{array}
$$
We rewrite this as
$$
\begin{array}{l}
\frac{(\dot f_1,\Delta  f_0, \Delta \dot f_0)}{(\dot f_1,\Delta
f_0, \Delta f_1)|\dot f_1{\times} \Delta f_0|^2}
\Big( (\Delta \dot f_0,\Delta f_1, \dot f_1 {\times}\Delta f_0){-}
(\dot f_1,\Delta f_0 {\times}\Delta \dot f_0,\Delta f_1){-}
(\dot f_1 {\times} \Delta \dot f_0,\Delta f_0, \Delta f_1)
\Big).
\end{array}
$$
Let us study the expression in the brackets.
$$
\begin{array}{l}
(\Delta \dot f_0,\Delta f_1, \dot f_1 {\times}\Delta f_0){-}
(\dot f_1,\Delta f_0 {\times}\Delta \dot f_0,\Delta f_1){-}
(\dot f_1 {\times} \Delta \dot f_0,\Delta f_0, \Delta f_1)=
\\
-\big(\Delta\dot f_0{\times}(\dot f_1 {\times}\Delta f_0)+
\dot f_1{\times}( \Delta f_0 {\times} \Delta\dot f_0)+
\Delta f_0{\times}(\Delta \dot f_0 {\times}\dot f_1 ),
\Delta f_1 \big)=(0,\Delta f_1)=0.
\end{array}
$$
The second equality holds by the Jacobi identity. Hence the
coefficient of $G_{12}$ is zero. Therefore,
$$
\Dd\(\dot f_0,\dot f_0\)=
2\(\Dd\Delta \dot f_0,\Delta \dot f_0\)=0,
$$
and $\(\dot f_0,\dot f_0\)$ is invariant under the infinitesimal
deformation.

The proof of the invariance of $\(\dot f_2,\dot f_2\)$ repeats the
proof for $\(\dot f_0,\dot f_0\)$.

So we have checked the invariance of all the 11 functions in the
definition of an infinitesimal flexion.
Hence $\Dd f$ is an infinitesimal flexion.
\end{proof}

Now we have all the ingredients to prove the main theorem of this subsection.

\subsubsection{Conclusion of the proof of Theorem~\ref{1-d-inf}}
{\it Existence.} The existence of an infinitesimal flexion follows directly from Proposition~\ref{existence}$(i)$.

\vspace{2mm}

{\noindent
{\it Uniqueness.} By Proposition~\ref{condition} every infinitesimal flexion satisfies System~A.
Since we consider 2-ribbon surfaces with fixed initial position, for every non-zero infinitesimal flexion $\Dd f$ we have:
$$
\Dd\dot{ f_1}(a)=0,  \quad \Dd \Delta f_0(a)=0, \quad  \hbox{and}
\quad \Dd \Delta f_1(a)=\alpha \dot f_1(a){\times} \Delta f_1(a)
$$
for some non-zero $\alpha$.
Hence by Proposition~\ref{existence} this is one of the flexions of Proposition~\ref{existence}$(i)$.
So the set of infinitesimal flexions is one-dimensional. Since the set is a linear space, it is
a line.
Hence $f$ has one degree of  infinitesimal flexibility.
\qed
}

Theorem~\ref{1-d-inf} together with Proposition~\ref{existence} imply the following.
\begin{corollary}\label{000}
Let $f\in C_0^{1,2,1}([a,b],\r^3)$ be a weakly generic 2-ribbon surface with fixed initial position,
and let $\Dd f$ be its infinitesimal flexion satisfying
$$
\Dd\dot{ f_1}(a)=0,  \quad \Dd \Delta f_1(a)=0, \quad  \hbox{and} \quad \Dd \Delta f_0(a)=0,
$$
Then $\Dd f=0$.
\qed
\end{corollary}


\subsection{Variational operators of infinitesimal flexions}\label{var}

Let us fix an orthonormal basis $(e_1,e_2,e_3)$ in $\r^3$.
Denote by $\Omega^1_{3{\times} 3}$ the Banach space
$$
\big((C^1[a,b])^3\big)^3\cong (C^1[a,b])^9
$$
with the norm
$$
\|(h_{11},h_{12},\ldots,h_{33})\|=\max\limits_{1\le i,j \le 3}(\max(\sup
|h_{ij}|,\sup |\dot h_{ij}|)).
$$

Consider the following map
$$
Z:C^{1,2,1}([a,b],\r^3)\to\Omega^1_{3{\times}3},
$$
where for a 2-ribbon surface $f$  the image $Z(f)$ in the basis $(e_1,e_2,e_3)$ is defined as
$$
\begin{array}{r}
\dot f_1(t)=\big(h_{11}(t),h_{12}(t),h_{13}(t)\big),
\\
\Delta f_0(t)=\big(h_{21}(t),h_{22}(t),h_{23}(t)\big),
\\
\Delta f_1(t)=\big(h_{31}(t),h_{32}(t),h_{33}(t)\big).
\end{array}
$$

Note that every 2-ribbon surface $f$ is defined by
$\dot f_1$, $\Delta f_0$, and $\Delta f_1$ up to a translation.
So after fixing, say, $f_1(a)=(0,0,0)$ one has a bijection.

We say that a point $h=(h_{11}, h_{12},\ldots, h_{33})$ in $\Omega^1_{3{\times}3}$ is in {\it general
position} if the determinant
$$
\det \left(
\begin{array}{ccc}
h_{11}& h_{12} & h_{13}\\
h_{21}& h_{22} & h_{23}\\
h_{31}& h_{32} & h_{33}\\
\end{array}
\right)\ne 0
$$
for every $t\in [a,b]$. This condition obviously
corresponds to the weakly genericity condition, i.e., to
$$
(\dot f_1,\Delta f_0,\Delta f_1)\ne 0.
$$

Denote by $\Sigma_\Omega$ the set of all points $h$ that are not in general position.

\vspace{2mm}

\begin{definition}\label{nu}
Denote by ${\mathcal V}^{\pm}:[0,\Lambda]\times (\Omega^1_{3{\times}3}\setminus \Sigma_\Omega) \to
\Omega^1_{3{\times}3}$ two variational {\it operators of infinitesimal
flexion} in coordinates $(h_{11},h_{12},\ldots, h_{33})$:
\begin{align}\label{e11}
\begin{array}{ll}
{\mathcal V}^{\pm}_{l{-}1,m}(\lambda, h)=&
\displaystyle\frac{(e_m,\Delta f_0,\Delta f_1)}{(\dot f_1,\Delta
f_0,\Delta f_1)}G_{l{-}1,1}(h)+
\frac{(\dot f_1,e_m,\Delta f_1)}{(\dot f_1,\Delta f_0,\Delta
f_1)}G_{l{-}1,2}(h)+\\
& \displaystyle\frac{(\dot f_1,\Delta f_0, e_m)}{(\dot f_1,\Delta
f_0, \Delta f_1)}G_{l{-}1,3}(h).
\end{array}
\end{align}
for ($1\le l,m\le 3$). Here $G_{11}(h),G_{12}(h), \ldots, G_{33}(h)$ is a
solution of System~A at point $f$ with the initial conditions
corresponding to
$$
\Dd\dot{ f_1}(a)=0,  \quad \Dd \Delta f_0(a)=0, \quad  \hbox{and}
\quad \Dd \Delta f_1(a)=\pm \dot f_1(a){\times} \Delta f_1(a),
$$
i.e.,
\begin{equation}\label{normal_condition}
\begin{array}{lll}
G_{11}(a)=0, \quad &
G_{12}(a)=0, \quad &
G_{13}(a)=0, \\
G_{21}(a)=0, \quad &
G_{22}(a)=0, \quad &
G_{23}(a)=0,\\
G_{31}(a)=0, \quad &
G_{32}(a)=\pm(\dot f_1(a), \Delta f_0(a), \Delta f_1(a)), \quad &
G_{33}(a)=0.
\end{array}
\end{equation}
(Here we take ``$+$'' sign for ${\mathcal V}^+$ and ``$-$'' for ${\mathcal V}^-$.)
\end{definition}

Note that both
${\mathcal V}^{+}$ and ${\mathcal V}^{-}$ are autonomous operators, they do not depend on time parameter
$\lambda$.

It is important that the following statement holds.
\begin{proposition}
Let $h$ be a point of $\Omega^1_{3{\times}3}$ in general position and $\lambda \in [0,\Lambda]$. Then we have
$$
{\mathcal V}^{\pm}(\lambda,h) \in \Omega^1_{3{\times}3}.
$$
\end{proposition}

\begin{proof}
The proof is straightforward, all functions involved in Expression~(\ref{e11}) are continuously differentiable,
and hence both
$
{\mathcal V}^{+}(\lambda,h)
$
and
$
{\mathcal V}^{-}(\lambda,h)
$
are continuously differentiable.
\end{proof}

\begin{remark} Let us show in brief how to find the coordinates of
the  infinitesimal deformation $\Dd f$ in the basis $e_1, e_2, e_3$ satisfying
$$
\Dd f_1(a)=0, \qquad \Dd \dot f_1 (a)=0, \qquad
\Dd\Delta f_0(a)=0,
\quad \hbox{and} \quad
\Dd \Delta f_1(a)=\dot f_1(a){\times} \Delta f_1(a).
$$
First, one should solve System~A with the above initial data, then
substitute the obtained solution $(G_{11},G_{12},\ldots, G_{33})$ to
Equations~(\ref{e11}). Now we have the coordinates of $\Dd \dot
f_1$, $\Dd\Delta f_0$, and $\Dd\Delta f_1$. Having the additional
condition $\Dd f_1(a)=0$ one can construct $\Dd f_1$, $\Dd f_0$, and
$\Dd f_2$:
$$
\Dd f_1(t_0)=\int\limits_{a}^{t_0}\Dd \dot f_1(t)d(t),
\quad \Dd f_0=\Dd f_1 - \Dd \Delta f_0,
\quad \Dd f_2 =\Dd f_1 + \Dd \Delta f_1.
$$
\end{remark}

Further we will work in the following subspace of
$\Omega^1_{3{\times}3}$. Denote
$$
\tilde \Omega^1_{3{\times}3}= \big\{h\in \Omega^1_{3{\times}3}\big|h_{12}(a)=h_{13}(a)=h_{23}(a)=0\big\}.
$$
It is clear that $\tilde \Omega^1_{3{\times}3}$ is a Banach space itself.

We have the following important property of $\tilde \Omega^1_{3{\times}3}$.
\begin{proposition}
For every $\lambda \in [0,\Lambda]$ and $h\in \tilde\Omega^1_{3{\times}3}\setminus\Sigma_\Omega$
the subspace $\tilde \Omega^1_{3{\times}3}$ is an invariant space of the operators ${\mathcal V}^+ (\lambda, h)$ and ${\mathcal V}^- (\lambda, h)$.
\end{proposition}

\begin{proof}
From the conditions
$$
\Dd\dot{ f_1}(a)=0,  \quad  \hbox{and}
\quad
\Dd \Delta f_0(a)=0
$$
we have $G_{ij}(a)=0$ for all $i = 1,2$, and $j=1,2,3$. Hence by Expression~(\ref{e11})
$$
{\mathcal V}_{11}^{\pm} (\lambda, h)(a)={\mathcal V}_{12}^{\pm} (\lambda, h)(a)=\ldots ={\mathcal V}^{\pm}_{23} (\lambda, h)(a)=0
$$
for all $\lambda\in [0,\Lambda]$ and $h\in \Omega^1_{3{\times}3}$.
Therefore, for every $\lambda\in [0,\Lambda]$ and $h\in \tilde \Omega^1_{3{\times}3}\setminus\Sigma_\Omega$ we
have ${\mathcal V}^{\pm}(\lambda, h)\in \tilde \Omega^1_{3{\times}3}$.
\end{proof}

Finally we have the following important statement.

\begin{proposition}
The map $Z$ is a bijection of $\tilde \Omega^1_{3{\times}3}$ and $C^{1,2,1}_0([a,b],\r^3)$.
\end{proposition}

\begin{proof}
The inverse map $Z^{-1}(h)=(f_0,f_1,f_2)$ is defined as
$$
f_1(t_0)=\int\limits_{a}^{t_0}
\left(
{\begin{array}{c}
h_{11}(t)\\
h_{12}(t)\\
h_{13}(t)
\end{array}
}
\right)
dt,
\hbox{ }
f_0(t_0)=f_1(t_0)-
\left(
{\begin{array}{c}
h_{21}(t_0)\\
h_{22}(t_0)\\
h_{23}(t_0)
\end{array}
}
\right),
\hbox{ }
f_2(t_0)=
\left(
{\begin{array}{c}
h_{31}(t_0)\\
h_{32}(t_0)\\
h_{33}(t_0)
\end{array}
}
\right)
-f_1(t_0).
$$
at every $t_0\in [a,b]$.
\end{proof}


\subsection{Finite flexibility of 2-ribbon surfaces}\label{finit2r}

In Subsection~\ref{inf_flex_subs} we showed that every $2$-ribbon surface in
general position is infinitesimally flexible and that the space of its
infinitesimal flexions is one-dimensional. The aim of this
subsection is to show that a weakly generic 2-ribbon surface
is finitely flexible and has one degree of finite flexibility.

\vspace{2mm}

\subsubsection{Lipschitz condition}
We start with the discussion of the initial value problem for the
following two differential equations on the set of all points
$\tilde\Omega^1_{3{\times}3}$ in general position (here $\lambda$ is the time
parameter):
\begin{equation}\label{mathcalV}
\frac{\partial h}{\partial \lambda}={\mathcal V}^{+} (\lambda, h)
\quad \hbox{and} \quad
\frac{\partial h}{\partial \lambda}={\mathcal V}^{-} (\lambda, h)
.
\end{equation}
To solve the initial value problem we study local Lipschitz
properties for ${\mathcal V}^{+}$ and ${\mathcal V}^{-}$.

\begin{definition}\label{defODE}
Consider a Banach space $E$ with a norm $| *|_E$, and a positive real number $\Lambda$. Let $U$ be a
subset of $[0,\Lambda]\times E$. We say that a functional
${\mathcal F}:U\to E$ {\it locally satisfies a Lipschitz
condition} if for every point $(\lambda_0,p)$ in $U$ there exist a
neighborhood $V$ of the point and a constant $K$ such that for every
pair of points $(\lambda,p_1)$ and $(\lambda,p_2)$ in $V$ the
inequality
$$
| {\mathcal F}(\lambda,p_1)-{\mathcal F}(\lambda,p_2)|_E\le
K|p_1-p_2 |_E
$$
holds.
\end{definition}

\vspace{2mm}

First we verify a Lipschitz condition for the following operator.
Define ${\mathcal G}:[0,\Lambda]\times \tilde\Omega^1_{3{\times}3} \to \tilde\Omega^1_{3{\times}3}$
by
$$
\quad {\mathcal G}_{ij} (\lambda, h)=G_{ij}(h), \quad 1 \le i,j\le 3,
$$
where $G_{ij}(h)$ are defined by Equations~(\ref{g_i}).

\begin{lemma}\label{g1}
For every point $h\in U$ in general position,
there exists a neighborhood $V_h$ of $h$ such that
the functional ${\mathcal G}$ locally satisfies a Lipschitz
condition in $[0,\Lambda] \times V_h$.
\end{lemma}

\begin{proof}
Consider a point $h\in U$. The element $(G_{11},G_{12},\ldots, G_{33})$ itself
satisfies a system of linear differential equations (System~A).
The coefficients of this system depend only on a point of
$\tilde\Omega_{3{\times}3}^1$. Since the point $h$ is in general position, there
exists an integer constant $K$ such that for a sufficiently small
neighborhood $V_h$ of $h$ the dependence is $K$-Lipschitz, i.e.,
for $p$ and $q$ from $V_h$ every coefficient $c$ of System~A satisfies the
inequality
$$
|c(p)-c(q)|<K \|p-q\|.
$$
Hence the solutions for $t\in [a,b]$ satisfy the Lipschitz
condition for a fixed initial data on $V_h$. (This is clear from the fact that
the solution of the system with small coefficients $c(p)-c(q)$
will be almost constant, the difference in each coordinate will not be greater than $9(b-a)K\|p-q\|$.)
Finally the solution for $t\in [a,b]$ satisfies the Lipschitz
condition for a fixed parameter and different initial data on $V_h$ (See Proposition~1.10.1 in~\cite{Car}).
Therefore, for some constants $\bar K_l$ we have
$$
\sup(|G_{ij}(p)-G_{ij}(q)|)<\bar K_{ij} \|p-q\|, \quad 1\le i,j\le 3.
$$
From System~A we know that the $\dot G_{i,j}$ linearly depend on $G_{11},G_{12},
\ldots, G_{33}$, therefore, we get the Lipschitz condition for the
derivatives: for some constants $\tilde K_l$ we have
$$
\sup(|\dot G_{ij}(p)-\dot G_{ij}(q)|)<\tilde K_{ij} \|p-q\|, \quad
1\le i,j \le 3.
$$
Thus there exists a real number $\hat K$ such that for all points
$p$ and $q$ in $V_h$,
$$
\|{\mathcal G}(\lambda, p)- {\mathcal G}(\lambda,
q)\|=\max\limits_{1\le i,j \le 3}\big(\max\big(\sup
|G_{ij}(p)-G_{ij}(q)|,\sup |\dot G_{ij}(p)-\dot G_{ij}(q)|\big)\big)<{\hat K}
\|p-q\|.
$$
Thus ${\mathcal G}$ satisfies a Lipschitz condition on
$V_h$.
Therefore, ${\mathcal G}$ satisfies a Lipschitz condition on
$[0,\Lambda] \times V_h$ (since ${\mathcal G}$ is autonomous).
\end{proof}

Lemma~\ref{g1} and Expression~(\ref{e11}) directly imply the
following statement.

\begin{corollary}\label{d1}
For every point $h\in U$ in general position and,
there exists a neighborhood $V_h$ of $h$ such that
both functionals ${\mathcal V}^{+}$ and ${\mathcal V}^{-}$ locally satisfy a Lipschitz
condition in $[0,\Lambda] \times V_h$. \qed
\end{corollary}


\subsubsection{Existence and uniqueness of solutions}
Let us prove the following general statement.

\begin{proposition}\label{euV+}
Let $h_0\in\tilde \Omega^1_{3{\times}3}$ be in general position.
Then for sufficiently small positive $\varepsilon$ there exists a unique solution $\gamma$ of the equation
\begin{equation}\label{eqqqqType2}
\frac{\partial h}{\partial \lambda}=
{\mathcal V}^{+}(\lambda,h)
\end{equation}
on $[-\varepsilon,\varepsilon]$, such that $\gamma(0)=h_0$.
\end{proposition}

We start with the following general lemma.
\begin{lemma}\label{euV+Lemma}
Let $h_0\in\tilde \Omega^1_{3{\times}3}$ be in general position.
A deformation $\gamma$ with $\gamma(0)=h_0$ is a solutions of Equation~$($\ref{eqqqqType2}$)$
if and only if $\gamma$ satisfies
$$
\frac{\partial \gamma}{\partial \lambda}
=
\left\{
\begin{array}{ll}
{\mathcal V}^{+} (\lambda, \gamma(\lambda)) &\hbox{ for  all $\lambda \in [0,\Lambda]$},
\\
-{\mathcal V}^- (-\lambda, \gamma(\lambda)) & \hbox{ for  all $\lambda \in [-\Lambda,0]$}.
\end{array}
\right.
$$
\end{lemma}

\begin{proof}
The proof of the Lemma is straightforward.
\end{proof}

{
\noindent {\it Proof of Proposition~\ref{euV+}.}
As we showed in Corollary~\ref{d1}, the operators ${\mathcal V^+}$ and ${\mathcal V^-}$
satisfy a Lipschitz condition in some neighborhood of the point $Z^{-1}(f)$.
From the general theory of differential equations on Banach spaces (see
for instance the first section of the second chapter
of~\cite{Car}) it follows that this condition implies local
existence and uniqueness of a solution of the initial value
problem for the differential Equations~(\ref{mathcalV}).
Hence by Lemma~\ref{euV+Lemma}
for a sufficiently small positive $\varepsilon$ there exists a unique solution
$\gamma$ of Equation~(\ref{eqqqqType2}) on $[-\varepsilon,\varepsilon]$ satisfying the condition $\gamma(0)=h_0$.
\qed
}


\subsubsection{Finite flexibility}
The key point for finite flexibility of 2-ribbon surfaces is the following lemma.

\begin{lemma}\label{Zgamma}
Let $\{f^\lambda\}$, $\lambda \in [-\varepsilon, \varepsilon]$ for some $\varepsilon>0$,
be a normalized isometric deformation in the space $C^{1,2,1}_0([a,b],\r^3)$
such that all 2-ribbon surfaces of the family are weakly generic.
Let also $\{Z^{-1}(f^\lambda)\}$
be the corresponding deformation in $\tilde \Omega^1_{3{\times}3}$.
Then $\{f^\lambda\}$ is an isometric deformation if and only $\{Z^{-1}(f^\lambda)\}$ satisfies
Equation$($\ref{eqqqqType2}$)$ for  all $\lambda \in [-\varepsilon,\varepsilon]$.
\end{lemma}

\begin{proof}
Let $\gamma=\{f^\lambda\}$, $\lambda \in [-\varepsilon, \varepsilon]$, be a normalized isometric deformation in $C^{1,2,1}_0([a,b],\r^3)$
such that all 2-ribbon surfaces of the family are weakly generic.
Every normalized deformation satisfies Condition~(\ref{normal_condition})
at every point $\lambda\in [-\varepsilon,\varepsilon]$ with the positive choice of the sign.

Since $\gamma$ is an isometric deformation, $\Dd_\gamma f^\lambda$ is an infinitesimal flexion at every point $\lambda\in [-\varepsilon, \varepsilon]$.
Hence the corresponding functions $G_{ij}^\lambda$ satisfy system~A (by Proposition~\ref{condition}).
Let us now write $\D\dot f_1^\lambda$, $\D \Delta f_0^\lambda$, and $\D\Delta f_1^\lambda$ in the basis $e_1$, $e_2$, $e_3$
using functions $G_{ij}^\lambda$. Recall that
$$
\D\dot f_1^\lambda=
(\D \dot f_1^\lambda, \dot f_1^\lambda) \dot f_1^\lambda+
(\D \dot f_1^\lambda, \Delta f_0^\lambda) \Delta f_0^\lambda+
(\D \dot f_1^\lambda, \Delta f_1^\lambda) \Delta f_1^\lambda
=
G_{11}\dot f_1^\lambda+G_{12}\Delta f_0^\lambda+G_{13}\Delta f_1^\lambda.
$$
Hence in the basis $(e_1,e_2,e_3)$ we have
\begin{equation}\label{eee}
\begin{array}{l}
\displaystyle
\frac{\partial \dot f_1^\lambda}{\partial\lambda}=\D\dot f_1^{\lambda}\\
\displaystyle
\quad\quad
=\sum\limits_{m=1}^{3}
\Bigg(
\frac{(e_m,\Delta f_0^\lambda,\Delta f_1^\lambda)}{(\dot f_1^\lambda,\Delta f_0^\lambda,\Delta f_1^\lambda)}G_{11}^\lambda+
\frac{(\dot f_1^\lambda,e_m,\Delta f_1^\lambda)}{(\dot f_1^\lambda,\Delta f_0^\lambda,\Delta f_1^\lambda)}G_{12}^\lambda+
\frac{(\dot f_1^\lambda,\Delta f_0^\lambda, e_m)}{(\dot f_1^\lambda,\Delta f_0^\lambda, \Delta f_1^\lambda)}G_{13}^\lambda
\Bigg)e_m\\
\quad\quad
=\sum\limits_{m=1}^{3}
\Big({\mathcal V}_{1,m}^+ \big(\lambda, Z^{-1}(f^\lambda)\big)\Big)e_m.
\end{array}
\end{equation}
Similarly we have:
\begin{equation}\label{eee2}
\begin{array}{l}
\displaystyle
\frac{\partial \Delta f_0^\lambda}{\partial\lambda}=
\sum\limits_{m=1}^{3}
\Big({\mathcal V}_{2,m}^+ \big(\lambda, Z^{-1}(f^\lambda)\big)\Big)e_m,
\\
\displaystyle
\frac{\partial \Delta f_1^\lambda}{\partial\lambda}=
\sum\limits_{m=1}^{3}
\Big({\mathcal V}_{3,m}^+ \big(\lambda, Z^{-1}(f^\lambda)\big)\Big)e_m.
\end{array}
\end{equation}


Hence by Definition~\ref{nu} the corresponding derivatives in the space $\tilde\Omega^1_{3{\times}3}$
satisfy:
$$
\frac{\partial Z^{-1}(f^\lambda)}
{\partial \lambda}
={\mathcal V}^{+} (\lambda, Z^{-1}(f^\lambda))
$$
for every $\lambda\in [-\varepsilon,\varepsilon]$.

\vspace{2mm}

Conversely, let $Z^{-1}(f^\lambda)$ satisfy
$$
\frac{\partial Z^{-1}(f^\lambda)}
{\partial \lambda}
={\mathcal V}^+ (\lambda, Z^{-1}(f^\lambda))
$$
for every $\lambda\in [-\varepsilon,\varepsilon]$.
Then the corresponding $\D\dot f_1^\lambda$, $\D\Delta f_0^\lambda$, and $\D \Delta f_1^\lambda$
are defined as in~(\ref{eee}) and~(\ref{eee2}).
Hence the correspondent scalar products $G_{ij}^\lambda$ satisfy System~A for $\lambda\in [-\varepsilon,\varepsilon]$.
Thus $\D f^{\lambda}$ is an infinitesimal flexion for every $\lambda \in [-\varepsilon,\varepsilon]$.
Hence by $\{f^{\lambda}\}$ is an isometric deformation on $[-\varepsilon,\varepsilon]$.
From construction it follows that $\{f^{\lambda}\}$ is a normalized deformation.
\end{proof}

Now we prove the following theorem on finite flexibility of
weakly generic 2-ribbon surfaces.

\begin{theorem}\label{2-ribbon flex}
Every 2-ribbon weakly generic semidiscrete surface $f$ in $C^{1,2,1}_0([a,b],\r^3)$ has one degree of finite flexibility.
\end{theorem}

\begin{proof}
On the one hand,
by Lemma~\ref{Zgamma} normalized isometric deformations of $f$ with a fixed initial position
are in one-to-one correspondence with Solution of~Equation~(\ref{eqqqqType2}) satisfying $\gamma(0)=Z^{-1}(f)$.
On the other hand,
by Proposition~\ref{euV+} for sufficiently small positive $\varepsilon$ there exists a unique solution $\gamma$
of Equation~(\ref{eqqqqType2}) satisfying $\gamma(0)=Z^{-1}(f)$.
Hence,
there exists a unique normalized isometric deformations of $f$ (with the parameter in $[-\varepsilon,\varepsilon]$
for sufficiently small positive $\varepsilon$).
Therefore, $f$ has one degree of finite flexibility.
\end{proof}

\begin{remark}
In fact, one can prove the statement of Theorem~\ref{2-ribbon flex} for the spaces of functions $C^{m,m{+}1,m}_0([a,b],\r^3)$
for arbitrary $m\ge 1$. We are not going to use this later so we omit the details here.
The proofs mostly repeat the ones for the case $m=1$ shown in details above.
\end{remark}


\section{Infinitesimal flexibility of $3$-ribbon surfaces}

In this section we find necessary infinitesimal flexibility condition of
$3$-ribbon surfaces. For the case of $n$-ribbon surfaces each 3-ribbon subsurface gives
a condition of infinitesimal flexibility.

\subsection{Preliminary statements on infinitesimal flexion of 3-ribbon surfaces}
In this subsection we prove certain relations that we use further
in the proof of the statement on infinitesimal flexibility
condition for 3-ribbon surfaces.


Consider the following function
$$
\Phi=\(\Delta f_0,\Delta f_1\).
$$
This function plays a central role in our further description of
the infinitesimal flexibility condition of $3$-ribbon and $n$-ribbon surfaces
(see Theorem~\ref{inf} and Theorem~\ref{n-ribbon-infini}). Let
$\Dd\Phi$ be the infinitesimal flexion of $\Phi$. Via the function $\Dd \Phi$
we describe monodromy conditions for finite flexibility.
Proposition~\ref{inner} and Corollary~\ref{discr} deliver
necessary tools to describe continuous and discrete parts of the
monodromy condition on $\Phi$.

\begin{remark}
In the proofs of the statements of this subsection
we fix the flexion of the initial frame at $t=a$ in the following way
$$
\Dd \dot f_1=\Dd \Delta f_1(t_0)=0
$$
(compare to the space $C^M_0{[a,b],\r^3}$ where $\Dd\dot f_1(a)=\Dd \Delta f_0(a)=0$ instead).
This simplifies calculations for the 3-ribbon surfaces, since the fixed bar with endpoints $f_1(a)$
and $f_2(a)$ belongs to the middle strip.
\end{remark}

\subsubsection{Continuous shift}

Here we study the dependence of the infinitesimal flexion $\Dd
\Phi$ on the argument $t$.

\begin{proposition}\label{inner} {\bf (On continuous shift.)}
Let $f$ be a weakly generic 2-ribbon surface in $C^{1,2,1}([a,b],\r^3)$.
Then for every infinitesimal flexion $\Dd \Phi$ the following condition holds:
$$
\Dd\Phi(t_2)=\Dd\Phi(t_1)\cdot\exp{\left(\int\limits_{t_1}^{t_2}\frac
{(\dot f_1,\Delta \dot f_0,\Delta f_1)
+(\dot f_1, \Delta f_0,\Delta \dot f_1)}
{(\dot f_1,\Delta f_0,\Delta f_1)}dt\right)}.
$$
\end{proposition}

This is a direct consequence of the next lemma.

\begin{lemma}~\label{rel}
Let $f$ be a weakly generic 2-ribbon surface in $C^{1,2,1}([a,b],\r^3)$,
then
$$
\Dd\dot{\Phi}=\frac
{(\dot f_1,\Delta \dot f_0,\Delta f_1)
+(\dot f_1, \Delta f_0,\Delta \dot f_1)}
{(\dot f_1,\Delta f_0,\Delta f_1)}\Dd\Phi.
$$
\end{lemma}

\begin{proof}
\begin{subequations}
Note that
$$
\begin{array}{l}
\Dd\Phi=\(\Dd \Delta f_0,\Delta f_1\)+\(\Delta f_0,\Dd \Delta f_1\), \qquad \hbox{and}\\
\Dd\dot{\Phi}=\(\Dd\Delta \dot f_0,\Delta f_1\)+\(\Dd \Delta
f_0,\Delta \dot f_1\)+
\(\Delta \dot f_0,\Dd \Delta f_1\)+\(\Delta f_0,\Dd \Delta \dot
f_1\).
\end{array}
$$
Let us prove the statement of the lemma for an arbitrary point
$t_0$. Without loss of generality we fix $\Dd\dot f_1(t_0)=0$ and
$\Dd \Delta f_1(t_0)=0$ (this is possible since every flexion is
isometric to a flexion with such properties and isometries of
flexions do not change the functions in the formula of the lemma).
Then $\Dd \Delta f_0(t_0)$ is proportional to $\dot
f_1(t_0){\times} \Delta f_0(t_0)$, and hence there exists some
real number $\alpha$ with
$$
\Dd \Delta f_0(t_0)=\alpha \dot f_1(t_0){\times} \Delta f_0(t_0).
$$
Thus we immediately get
$$
\Dd\Phi(t_0)=\big\(\Dd \Delta f_0(t),\Delta f_1(t)\big\)=\alpha
\big(\dot f_1(t_0),\Delta f_0(t_0),\Delta f_1(t_0) \big).
$$
Let us express the summands for $\Dd\dot{\Phi}(t_0)$. We start with
$\big\(\Dd \Delta \dot f_0(t_0),\Delta f_1(t_0)\big\)$. First we
note that
\begin{equation}
\Delta f_1=\frac{(\Delta f_1,\Delta f_0,f_1{\times}\Delta \dot
f_0)}{(\dot f_1,\Delta f_0,\dot f_1{\times}\Delta f_0)}\dot f_1+
\frac{(\dot f_1,\Delta f_1,f_1{\times}\Delta \dot f_0)}{(\dot
f_1,\Delta f_0,\dot f_1{\times}\Delta f_0)}\Delta f_0+
\frac{(\dot f_1,\Delta f_0,\Delta f_1)}{(\dot f_1,\Delta f_0, \dot
f_1{\times}\Delta f_0)}f_1{\times}\Delta \dot f_0.
\label{ei}\end{equation}
Equation~(\ref{e8}) implies
\begin{equation}
\big\( \Dd\Delta \dot f_0(t_0),\dot f_1(t_0) \big\)=
-\big\( \Dd\dot f_1(t_0),\Delta \dot f_0(t_0) \big\) =
-\big\(0,\Delta \dot f_0(t_0) \big\) =0. \label{eii}
\end{equation}
From Equation~(\ref{e4}) we have
\begin{equation}
\big\( \Dd\Delta \dot f_0(t_0),\Delta f_0(t_0) \big\)=
-\big\(\Dd \Delta f_0(t_0),\Delta \dot f_0(t_0)\big\)= -\alpha
\big(\dot f_1 (t_0),\Delta f_0(t_0),\Delta \dot f_0(t_0) \big).
\label{eiii}\end{equation}
The function $(\Delta \dot f_0,\dot f_1,\Delta f_0)$ is invariant
of an infinitesimal flexion, therefore:
$$
(\Dd\Delta \dot f_0,\dot f_1, \Delta f_0)+(\Delta \dot f_0,\Dd\dot
f_1, \Delta f_0)+(\Delta \dot f_0 ,\dot f_1, \Dd \Delta f_0)=0,
$$
and hence
\begin{equation}
\begin{array}{r}
\big\(\Dd\Delta \dot f_0(t_0), \dot f_1(t_0){\times}\Delta
f_0(t_0)\big\)= -\big(\Delta \dot f_0(t_0), \dot f_1(t_0), \Dd
\Delta f_0(t_0)\big)\qquad \quad
\\
=-\alpha\big(\Delta \dot f_0(t_0),\dot f_1(t_0),\dot
f_1(t_0){\times}\Delta f_0(t_0) \big). \end{array}
\label{eiv}\end{equation}
Now we decompose $\Delta \dot f_0(t_0)$ in the last formula in the
basis of vectors $\dot f_1(t_0)$, $\Delta f_0(t_0)$, and $\Delta
f_1(t_0)$:

\begin{align*}
\big(\Delta \dot f_0(t_0),\dot f_1(t_0),\dot
f_1(t_0){\times}\Delta f_0(t_0) \big){=}&
\begin{array}{l}
\frac{\big(\dot f_1(t_0), \Delta \dot f_0(t_0),\Delta f_1(t_0)
\big)}{\big(\dot f_1(t_0), \Delta f_0(t_0),\Delta f_1(t_0) \big)}
\big(\Delta f_0(t_0),\dot f_1(t_0),\dot f_1(t_0){\times}\Delta
f_0(t_0) \big){+}
\end{array}
\\&
\begin{array}{l}
\frac{\big(\dot f_1(t_0), \Delta f_0(t_0), \Delta \dot f_0(t_0)
\big)}{\big(\dot f_1(t_0), \Delta f_0(t_0),\Delta f_1(t_0) \big)}
\big(\Delta f_1(t_0),\dot f_1(t_0),\dot f_1(t_0){\times}\Delta
f_0(t_0) \big).
\end{array}
\end{align*}
Therefore, after substitution~(\ref{ei}) of $\Delta f_2$ we
apply~(\ref{eii}), (\ref{eiii}), (\ref{eiv}), and the last
expression and get
\begin{align*}
\big\( \Dd \Delta \dot f_0 (t_0),\Delta f_1(t_0)\big\)=&
\begin{array}{l}
-\alpha \frac{\big(\dot f_1(t_0),\Delta f_1(t_0),\dot
f_1(t_0){\times}\Delta f_0(t_0)\big)}{\big(\dot f_1(t_0),\Delta
f_0(t_0),\dot f_1(t_0){\times}\Delta f_0(t_0)\big)} \big(\dot
f_1(t_0),\Delta f_0(t_0),\Delta \dot f_0(t_0) \big)-
\end{array}
\\&
\begin{array}{l}
\alpha \frac{\big(\dot f_1(t_0), \Delta \dot f_0(t_0),\Delta
f_1(t_0) \big)}{\big(\dot f_1(t_0),\Delta f_0(t_0),\dot
f_1(t_0){\times}\Delta f_0(t_0)\big)}
\big(\Delta f_0(t_0),\dot f_1(t_0),\dot f_1(t_0){\times} \Delta
f_0(t_0) \big)-
\\
\alpha \frac{\big(\dot f_1(t_0), \Delta f_0(t_0), \Delta \dot
f_0(t_0) \big)}{\big(\dot f_1(t_0),\Delta f_0(t_0),\dot
f_1(t_0){\times}\Delta f_0(t_0)\big)}
\big(\Delta f_1(t_0),\dot f_1(t_0),\dot f_1(t_0){\times}\Delta
f_0(t_0) \big)
\\
\end{array}
\\
=&-\alpha \big(\dot f_1(t_0),\Delta f_1(t_0),\Delta \dot
f_0(t_0)\big).
\\
\end{align*}
Similar calculations for the summand $\big\(\Delta
f_0(t_0),\Dd\Delta \dot f_1(t_0)\big\)$ (applying
Equations~(\ref{e3}), (\ref{e5}), and~(\ref{e9}) and the
conditions $\Dd\dot f_1(t_0)=0$ and $\Dd \Delta f_1(t_0)=0$) show
that
$$
\big\(\Delta f_0(t_0),\Dd\Delta \dot f_1(t_0)\big\)=0.
$$
Further we have
$$
\begin{array}{l}
\big\(\Dd \Delta f_0(t_0),\Delta \dot f_1(t_0)\big\)=\alpha\big(\dot f_1(t_0),\Delta f_0(t_0),\Delta \dot f_1(t_0)\big),\\
\big\(\Delta \dot f_0(t_0), \Dd \Delta f_1(t_0)\big\)=0.
\end{array}
$$
Therefore,
$$
\Dd\dot{\Phi}(t_0)=\alpha
\big(
\big(\dot f_1(t_0),\Delta \dot f_0(t_0),\Delta f_1(t_0) \big)+
\big(\dot f_1(t_0),\Delta f_0(t_0),\Delta \dot f_1(t_0)\big)
\big),
$$
and consequently
$$
\Dd\dot{\Phi}(t_0)=\frac
{\big(\dot f_1(t_0),\Delta \dot f_0(t_0), \Delta f_1(t_0)\big)
+\big(\dot f_1(t_0), \Delta f_0(t_0),\Delta \dot f_1(t_0) \big)}
{\big(\dot f_1(t_0),\Delta f_0(t_0),\Delta f_1(t_0)
\big)}\Dd\Phi(t_0).
$$
Thus Lemma~\ref{rel} holds for all possible values of $t_0$.
\end{subequations}
\end{proof}

\subsubsection{Discrete shift}

Every 3-ribbon surface contains two 2-ribbon surfaces as a subsurfaces.
Each of them has an infinitesimal flexion $\Dd\Phi_i$ ($i=1,2$),
where
$$
\Phi_1=\(\Delta f_0,\Delta f_1\) \qquad \hbox{and} \qquad
\Phi_2=\(\Delta f_1,\Delta f_2\).
$$
Let us show the relation between $\Dd\Phi_1$ and $\Dd\Phi_2$ for
the same values of argument $t$.

First, in Proposition~\ref{curvatures} we show a relation for
$\Dd\(\ddot f_1, \ddot f_1\)$ and $\Dd\(\ddot f_2, \ddot f_2\)$.
Second, in Proposition~\ref{PhiPsi} we give a link between
$\Dd\(\ddot f_1, \ddot f_1\)$ and $\Dd\Phi_1$. This will result in
the formula of Corollary~\ref{discr} on the relation between
$\Dd\Phi_1$ and $\Dd\Phi_2$.

\vspace{2mm}

We start with a formula expressing $\Dd\(\ddot f_2, \ddot f_2\)$
via $\Dd\(\ddot f_1, \ddot f_1\)$.

\begin{proposition}\label{curvatures}
Let $f$ be a strongly generic 3-ribbon surface in $C^{1,2,2,1}([a,b],\r^3)$, and let $\Dd f$ be its infinitesimal flexion.
Then the following equation holds:
$$
\Dd\(\ddot f_2, \ddot f_2\)= \frac{(\dot f_2, \ddot f_2,\Delta
f_1)}{(\dot f_1, \ddot f_1,\Delta f_1)}\Dd\(\ddot f_1, \ddot f_1\).
$$
\end{proposition}

\begin{proof}
We do calculations at a point $t_0$ again assuming that $\Dd\dot
f_1(t_0)=0$ and $\Dd \Delta f_1(t_0)=0$ (by choosing an appropriate
isometric representative of the deformation). Let us show that
$\Dd\dot f_2 (t_0)=0$. First, note that
$$
\Dd\dot f_2(t_0)=\Dd\dot f_1(t_0)+\Dd\Delta \dot f_1(t_0)=\Dd\Delta
\dot f_1(t_0).
$$
Secondly we show that the inner products of $\Dd\Delta \dot
f_1(t_0)$ and the vectors $\dot f_1(t_0)$, $\Delta f_1(t_0)$, and
$\dot f_1(t_0){\times}\Delta f_1(t_0)$ are all zero (this would
imply that $\Dd\Delta \dot f_1  (t_0)=0$).

From Equation~(\ref{e9}) we have
$$
\big\(\Dd\Delta \dot f_1(t_0),\dot f_1(t_0)\big\)=-\big\(\Dd\dot
f_1(t_0),\Delta \dot f_1(t_0)\big\)=
-\big\(0,\Delta \dot f_1(t_0)\big\) =0.
$$
Further, from Equations~(\ref{e5}), we get
$$
\big\(\Dd\Delta \dot f_1(t_0),\Delta f_1(t_0)\big\)=-\big\(\Dd
\Delta f_1(t_0),\Delta \dot f_1(t_0)\big\)=0.
$$
Finally, from the equation $\Dd (\dot f_1, \Delta f_1, \Delta \dot
f_1)=0$ we obtain
$$
\begin{array}{l}
\big\(\Dd\Delta \dot f_1(t_0),\dot f_1(t_0){\times}\Delta
f_1(t_0)\big\)=
\\
-\big(\Delta \dot f_1(t_0),\Dd\dot f_1(t_0),\Delta f_1(t_0)\big) -
\big(\Delta \dot f_1(t_0),\dot f_1(t_0),\Dd \Delta f_1(t_0)\big)=0.
\end{array}
$$
Therefore, $\Dd\Delta \dot f_1(t_0)=0$, and hence  $\Dd\dot
f_2(t_0)=0$.

\vspace{2mm}

From Equation~(\ref{e1}) and Equation~(\ref{e7}) we get
$$
\begin{array}{l}
\big\( \Dd\ddot f_1(t_0), \dot f_1(t_0) \big\)=
\frac{\partial}{\partial t} \big\( \Dd\dot f_1(t_0), \dot f_1(t_0)
\big\)-\big\( \ddot f_1(t_0), \Dd\dot f_1(t_0) \big\)=0-\big\(
\ddot f_1(t_0), 0 \big\)=0;\\
\big\(\Dd \ddot f_1(t_0),\Delta f_1(t_0)\big\)=
-\big\(\ddot f_1(t_0),\Dd\Delta f_1(t_0)\big\)=
-\big\(\ddot f_1(t_0),0\big\)=0.
\end{array}
$$
Therefore, for some real number $\beta_1$ we have
$$
\Dd\ddot f_1(t_0)=\beta_1 \dot f_1(t_0){\times}\Delta f_1(t_0).
$$
By a similar reasoning (since we have shown that $\Dd \dot
f_2(t_0)=0$) we get
$$
\Dd\ddot f_2(t_0)=\beta_2 \dot f_2(t_0){\times}\Delta f_1(t_0).
$$
 Since $\frac{\partial}{\partial t}\big(\Dd (\dot f_1, \Delta f_1,
\dot f_2)\big)=0$,  at point $t_0$  we have
$$
\big(\Dd\ddot f_1(t_0),\Delta f_1(t_0), \dot f_2(t_0)\big
)+\big(\dot f_1(t_0),\Delta f_1(t_0), \Dd\ddot f_2(t_0)\big )=0.
$$
Hence,
$$
\beta_1 \big(\dot f_1(t_0){\times}\Delta f_1(t_0),\Delta f_1(t_0),
\dot f_2(t_0)\big )+ \beta_2\big(\dot f_1(t_0),\Delta f_1(t_0),
\dot f_2(t_0){\times}\Delta f_1(t_0)\big )=0,
$$
and, therefore $\beta_1 =\beta_2$. This implies
$$
\Dd\big\(\ddot f_1(t_0), \ddot f_1(t_0) \big\)= 2\big\(\Dd\ddot
f_1(t_0), \ddot f(t_0) \big\)= 2 \beta_1 \big( \dot
f_1(t_0),\Delta f_1(t_0),\ddot f_1(t_0) \big)
$$
and
$$
\Dd\big\(\ddot f_2(t_0), \ddot f_2(t_0) \big\)= 2 \beta_1 \big(
\dot f_2(t_0),\Delta f_1(t_0),\ddot f_2(t_0) \big).
$$
The last two formulas imply the statement of
Proposition~\ref{curvatures}.
\end{proof}

Now let us relate $\Dd\(\ddot f_1, \ddot f_1\)$ and $\Dd\Phi$.

\begin{proposition}\label{PhiPsi}
Let $f$ be a weakly generic 2-ribbon surface in $C^{1,2,1}([a,b],\r^3)$.
Then the following identity holds:
$$
\Dd\(\ddot f_1, \ddot f_1\)=2
\frac{(\dot f_1, \ddot f_1, \Delta f_0)(\dot f_1, \ddot f_1,
\Delta f_1)}{(\dot f_1,\Delta f_0,\Delta f_1)^2}\Dd \Phi.
$$
\end{proposition}

\begin{proof}
We restrict ourselves to the case of a point. Without loss of
generality we assume that $\Dd\dot f_1(t_0)=0$ and $\Dd \Delta
f_1(t_0)=0$. So as we have seen before, there exists $\alpha$ such
that
$$
\Dd \Delta f_0(t_0)=\alpha \dot f_1(t_0){\times} \Delta f_0(t_0)
$$
and hence
$$
\Dd\Phi(t_0)=\alpha \big(\dot f_1(t_0),\Delta f_0(t_0),\Delta
f_1(t_0) \big).
$$
Let us calculate $\Dd\(\ddot f_1, \ddot f_1\)=2 \(\Dd\ddot f_1,
\ddot f_1\)$.
Decompose
$$
\ddot f_1=\frac{(\ddot f_1,\Delta f_0,\Delta f_1)}{(\dot
f_1,\Delta f_0,\Delta f_1)}\dot f_1+
\frac{(\dot f_1,\ddot f_1,\Delta f_1)}{(\dot f_1,\Delta f_0,\Delta
f_1)}\Delta f_0+
\frac{(\dot f_1,\Delta f_0, \ddot f_1)}{(\dot f_1,\Delta f_0,
\Delta f_1)}\Delta f_1.
$$
Since
$$
\big\(\Dd {\ddot f_1}(t_0),\dot f_1(t_0)\big\)=0, \qquad \hbox{
and} \qquad
\big\(\Dd {\ddot f_1}(t_0),\Delta f_1(t_0)\big\)=0,
$$
we get
$$
\Dd\(\ddot f_1(t_0), \ddot f_1(t_0)\)=2
\frac{\big(\dot f_1(t_0), \ddot f_1(t_0),\Delta
f_1(t_0)\big)}{\big(\dot f_1(t_0),\Delta f_0(t_0), \Delta
f_1(t_0)\big)}\big\(\Dd {\ddot f_1}(t_0),\Delta f_0(t_0)\big\).
$$
By Equation~(\ref{e6}) we have
$$
\big\(\Dd {\ddot f_1},\Delta f_0\big\)=-\big\( {\ddot f_1},\Dd
\Delta f_0\big\).
$$
Hence after the substitution of $\Dd \Delta f_0(t_0)$ in the first
summand one gets
$$
\big\(\Dd {\ddot f_1},\Delta f_0\big\)=\alpha \big(\dot
f_1(t_0),\ddot f_1(t_0),\Delta f_0(t_0) \big)=
\frac{\big(\dot f_1(t_0),\ddot f_1(t_0),\Delta f_0(t_0) \big)}
{\big(\dot f_1(t_0), \Delta f_0(t_0), \Delta f_1(t_0)\big)}
\Dd\Phi(t_0).
$$
Therefore, we obtain
$$
\Dd\(\ddot f_1(t_0), \ddot f_1(t_0)\)=2
\frac{\big(\dot f_1(t_0), \ddot f_1(t_0),\Delta
f_1(t_0)\big)\big(\dot f_1(t_0),\ddot f_1(t_0),\Delta f_0(t_0)
\big)} {\big(\dot f_1(t_0), \Delta f_0(t_0), \Delta
f_1(t_0)\big)^2}
\Dd\Phi(t_0).
$$
Since the statement does not depend on the choice of the basis and
invariant under isometries, we get the statement for all the
points.
\end{proof}

Let us show a formula of a discrete shift.

\begin{corollary}\label{discr}{\bf (On discrete shift.)}
Let $f$ be a strongly generic 3-ribbon surface in $C^{1,2,2,1}([a,b],\r^3)$.
Then the following holds:
$$
\Dd\Phi_2(t)=\frac{\big(\dot f_1(t),\ddot f_1(t),\Delta
f_0(t)\big)}{\big(\dot f_2(t),\ddot f_2(t),\Delta f_2(t)\big)}
\frac{\big(\dot f_2(t),\Delta f_1(t),\Delta
f_2(t)\big)^2}{\big(\dot f_1(t),\Delta f_0(t),\Delta
f_1(t)\big)^2}\Dd\Phi_1(t).
$$
\end{corollary}

\begin{proof}
The statement follows directly from Propositions~\ref{curvatures}
and~\ref{PhiPsi}.
\end{proof}

\subsection{Necessary condition of infinitesimal flexibility}

In this subsection we write down the infinitesimal flexibility
monodromy conditions for 3-ribbon surfaces (via continuous shifts
of Proposition~\ref{inner} and discrete shifts of
Corollary~\ref{discr}). Recall that
$$
\Lambda(t)= \frac{\big(\dot f_1(t),\ddot f_1(t),\Delta
f_0(t)\big)}{\big(\dot f_2(t),\ddot f_2(t),\Delta f_2(t)\big)}
\frac{\big(\dot f_2(t),\Delta f_1(t),\Delta
f_2(t)\big)^2}{\big(\dot f_1(t),\Delta f_0(t),\Delta
f_1(t)\big)^2},
$$
and
$$
H_i(t)= \frac{(\dot f_i(t), \Delta \dot f_{i-1}(t),\Delta
f_{i}(t))
+(\dot{f}_i(t), \Delta f_{i-1}(t), \Delta \dot f_{i}(t))}
{(\dot{f}_i(t),\Delta f_{i-1}(t),\Delta f_{i}(t))}, \quad i=1,2.
$$\label{1111}
\begin{theorem}\label{inf}
Let $f$ be a strongly generic 3-ribbon surface in $C^{1,2,2,1}([a,b],\r^3)$.
If the surface $f$ is infinitesimally flexible then
for every $t_1,t_2\in[a, b]$ we have
$$
\Lambda(t_2)\cdot \exp{\Big(\int\limits_{t_1}^{t_2}
H_1(t)dt\Big)} =
\Lambda(t_1)\cdot \exp{\Big(\int\limits_{t_1}^{t_2}H_2(t)dt\Big)}.
$$
\end{theorem}

\begin{remark}\label{infREM}
The condition of the proposition can be written in the ``almost'' equivalent infinitesimal form:
$$
\dot \Lambda-(H_2-H_1)\Lambda=0.
$$
Here the left hand side expression is considered as a function in the interval $[a,b]$.
The last expression has one disadvantage, $\dot \Lambda$ involves the third derivatives of $f_1$ and $f_2$,
while the expressions in proposition involve only up to the second derivatives.
\end{remark}

\begin{proof}
Let $f$ be infinitesimally flexible and $\Dd f$ be its infinitesimal nonzero flexion.
On the one side by Corollary~\ref{discr} we get relations between $\Dd\Phi_1(t_i)$
and $\Dd\Phi_2(t_i)$ for $i=1,2$. On the other side,
Proposition~\ref{inner} relates $\Dd\Phi_i(t_1)$ and
$\Dd\Phi_i(t_2)$ for $i=1,2$. These four relations define the
monodromy condition for $\Phi_i$ that is the condition in the
theorem, Therefore, it holds if a surface is infinitesimally
flexible.
\end{proof}

\begin{remark}
Let us write a more simple expressions for a surface
$w$ defined as
$$
\begin{array}{l}
w_0=f_1-\frac{1}{(\dot{f}_1,\Delta f_0,\Delta
f_{1})}\Delta f_0;\\
w_1=f_1;\\
w_2=f_2;\\
w_3=f_2+\frac{1}{(\dot{f}_2,\Delta f_1,\Delta f_2)}\Delta f_2.
\end{array}
$$
As one can see, all rulings of $w$ (if non-vanished) are parallel to the corresponding rulings of $f$.
In addition the middle strip of $f$ coincides with the middle strip of $w$.

Notice that
$$
\big(\dot{w}_1(t),\Delta w_0(t),\Delta w_{1}(t)\big)=1 \qquad
\hbox{and} \qquad \big(\dot{w}_2(t),\Delta w_1(t),\Delta
w_{2}(t)\big)=1
$$
for all arguments $t$. Hence we have:
$$
\begin{array}{l}
\Lambda= \frac{\big(\dot w_1,\ddot w_1,\Delta w_0\big)}{\big(\dot
w_2,\ddot w_2,\Delta w_2\big)},
\\
H_i= -(\ddot w_i,\Delta w_{i-1},\Delta w_{i}), \quad i=1,2.
\end{array}
$$
Note that this expression holds momentary.
\end{remark}

We conclude this subsection with the following open problem.

\begin{problem}
Find a sufficient condition for infinitesimal/finite flexibility of semidiscrete and
$3$-ribbon surfaces.
\end{problem}

\section{Flexibility of $n$-ribbon surfaces}

In this section we study flexibility questions for general case of $n\ge 2$.
We show that a strongly generic $n$-ribbon surface
has at most one degree of finite and infinitesimal flexibility (Subsection~4.1).
Further we study flexions of combined $n$-ribbon surfaces (Subsection~4.2).
This allows us to prove that
finite or infinitesimal flexibility of generic $n$-ribbon surfaces is identified by
finite or infinitesimal flexibility of all its 3-ribbon subsurfaces (Subsection~4.3).

\subsection{At most one degree of flexibility for strongly generic $n$-ribbon surfaces}

In this subsection we prove that all nontrivial infinitesimal flexions of strongly generic $n$-ribbon surfaces are
strongly isometrically nontrivial, and that such surfaces has at most one degree of infinitesimal flexibility.

\vspace{2mm}

Let us start with a useful tool to work with isometrically nontrivial flexions.

\begin{lemma}\label{lem-isometric}
An infinitesimal flexion of a weakly generic $n$-ribbon surface $f$ in the space $C^{0,1,0}([a,b],\r^3)$ is isometrically nontrivial
at a point $(t,i)$ $($where $i\in [1,\ldots, n-1]$$)$ if and only if
$$
\Dd\(\Delta f_{i-1}(t),\Delta f_i(t)\)\ne 0.
$$
\end{lemma}

\begin{proof}
Since $f$ is weakly generic, the pairs of vectors $(\dot f_i, \Delta f_{i-1})$ and $(\dot f_i, \Delta f_{i})$
span two non-coinciding 2-spaces $\pi_1$ and $\pi_2$.

Since $\pi_1$ and $\pi_2$ do not coincide, the condition $\Dd\(\Delta f_{i-1},\Delta f_i\)\ne 0$ is
equivalent to the fact that the infinitesimal flexion of the angle between $\pi_1$ and $\pi_2$ is non-zero.
Therefore, by Definition~\ref{def_isom} the last is equivalent to $f$ being isometrically nontrivial at a point $(t,i)$.
\end{proof}

In the next proposition we prove two important preliminary statements.

\begin{proposition}\label{iso-strong-iso}
Consider $n\ge 2$.
Let $f$ be a strongly generic $n$-ribbon surface in the space $C^{1,2,2,\ldots,2,1}_0([a,b],\r^3)$. Then the following two statements hold.

{\noindent
{\it $($i$)$}
Every isometrically nontrivial infinitesimal flexion of $f$ is strongly isometrically nontrivial
$($i.e., $f$ is isometrically nontrivial at every point $(t,i)$$)$.
}

{\noindent
{\it $($ii$)$}
For every regular isometric deformation $\gamma$ there exists
a locally monotone function $\xi$ such that $\gamma(\xi)$ is a normalized isometric deformation of $f$ in some neighborhood of $0$.
}
\end{proposition}

\begin{proof}
We prove Theorem~\ref{iso-strong-iso}$($i$)$ by induction in $n$.

{\noindent
{\it Base of induction. Case $n=2$.}
Let $\Dd f$ be a nontrivial infinitesimal flexion of a weakly generic 2-ribbon surface $f$ in $C^{1,2,1}_0([a,b],\r^3)$.
Therefore, there exists $t_0$ such that $\Dd\Phi(t_0)\ne 0$.
By Proposition~\ref{inner}, $\Dd\Phi(t_0)\ne 0$ implies that $\Dd\Phi(t)\ne 0$ for every $t\in[a,b]$.
Hence, by Lemma~\ref{lem-isometric} $f$ is isometrically nontrivial at each point $(t,1)$.
Therefore, $f$ is strongly isometrically nontrivial.
}

\vspace{2mm}

{\noindent
{\it Case $n=3$}.
 Let $\Dd f$ be a nontrivial infinitesimal flexion of a generic 3-ribbon surface $f$.
Therefore, there exists a point $(t_0,i)$ such that $\Dd\Phi_i(t_0)\ne 0$.
Without loss of generality we assume that $i=1$ (the case $i=2$ is similar).

By the above in case $n=2$ we have: $\Dd\Phi_1(t_0)\ne 0$ implies that $\Dd\Phi_1(t)\ne 0$ for every $t\in[a,b]$.
By Corollary~\ref{discr} (and the strongly generic condition for $f$), for every $t\in[a,b]$ the statement $\Dd\Phi_1(t)\ne 0$ implies that $\Dd\Phi_2(t)\ne 0$.
Hence, by Lemma~\ref{lem-isometric} $f$ is isometrically nontrivial
at each point $(t,1)$ and $(t,2)$.
Therefore, by Definition~\ref{def_isom} $f$ is strongly isometrically nontrivial.
}

\vspace{2mm}

{\noindent
{\it Step of induction}.
Consider a strongly generic $n$-ribbon surface $f$ with $n\ge 4$.
Denote
$$
f^1=(f_0,f_1,\ldots, f_{n-1}),
\qquad
f^2=(f_1,\ldots, f_{n-1},f_n)
\quad
\hbox{and}
\quad
f^{12}=(f_1,\ldots, f_{n-1})
$$
Let $\Dd f$ be isometrically nontrivial flexion of $f$.
Without loss of generality we assume that $\Dd f^1$ is isometrically nontrivial.
Hence by the induction assumption $\Dd f^1$ is strongly isometrically nontrivial.
Thus, $\Dd f^{12}$ is strongly isometrically nontrivial.
Since $f^{12}$ is a $(n-2)$-ribbon (with $n\ge 4$), we have that $\Dd f^2$ is isometrically nontrivial.
Then by the induction assumption $\Dd f^2$ is strongly isometrically nontrivial.
}

Since $n\ge 3$ and both $f^1$ and $f^2$ are strongly isometrically nontrivial,
$f$ is strongly isometrically nontrivial as well.
This concludes the proof of Proposition~\ref{iso-strong-iso}$($i$)$.

\vspace{2mm}

{
\noindent
Let us prove Proposition~\ref{iso-strong-iso}$($ii$)$.
Let $\{f^\lambda\}$ be a regular isometric deformation of $f$ with parameter $\lambda \in [-\Lambda,\Lambda]$.
Since $f$ has a fixed initial position, we have:
$$
\Dd_{f^\lambda} \dot f_1^{\lambda_0}(a)=0, \quad \Dd_{f^\lambda} \Delta f_0^{\lambda_0}(a)=0,
\quad \hbox{and} \quad
\Dd_{f^\lambda} \Delta f_1^{\lambda_0}(a)= \alpha(\lambda) \dot f_1^{\lambda_0}(a) {\times} \Delta f_1^{\lambda_0}(a)
$$
for every $\lambda_0\in[-\Lambda,\Lambda]$.
}

Since $\{f^\lambda\}$ is regular, $\Dd_{f^\lambda}f^{0}\ne 0$.
Therefore, there exists $\varepsilon >0$ such that $\Dd_{f^\lambda}f^{\lambda_0}\ne 0$ for $\lambda_0\in[-\varepsilon,\varepsilon]$.
From  Proposition~\ref{iso-strong-iso}$($i$)$ it follows that
for every $\lambda_0\in[-\varepsilon,\varepsilon]$ the flexion $\Dd_{f^\lambda}f^{\lambda_0}$ is strongly isometrically nontrivial.
Hence
$$
\alpha(\lambda)\ne 0 \quad \hbox{for} \quad \lambda\in[-\varepsilon,\varepsilon].
$$
Therefore, $\alpha(\lambda)$ is either a positive function or a negative function on $[-\varepsilon,\varepsilon]$. Denote
$$
\varphi(\lambda)=\int\limits_{0}^{\lambda}\alpha(\tau)d\tau.
$$
The function $\varphi$ is monotonous on $[-\varepsilon,\varepsilon]$, and hence there exists an inverse function $\varphi^{-1}$ on that interval.
Denote
$$
\xi=
\left\{
\begin{array}{rl}
\varphi^{-1}, & \hbox{if $\varphi$ is increasing,}\\
-\varphi^{-1}, & \hbox{if $\varphi$ is decreasing.}\\
\end{array}
\right.
$$
Choose positive $\hat \varepsilon$ such that $\xi$ is defined on $[-\hat \varepsilon,\hat \varepsilon]$.
Then
$$
\alpha(\xi(\lambda))=1
$$
for all $\lambda \in [-\hat \varepsilon,\hat \varepsilon]$. Hence
$\gamma\circ\xi$ is a normalized isometric deformation.
\end{proof}

Now we study degrees of finite and infinitesimal flexibility.

\begin{remark}\label{proportional2D}
To be consistent we mention the case of 2-ribbon surfaces.
Let $f$ be a weakly generic 2-ribbon semidiscrete surface in the space $C^{1,2,1}_0([a,b],\r^3)$.
Then the following two statements hold.
\\
{\it $($i$)$} The surface $f$ has one degree of  infinitesimal flexibility (Theorem~\ref{1-d-inf}).
\\
{\it $($ii$)$} The surface $f$ has one degree of finite flexibility (Theorem~\ref{2-ribbon flex}).
\end{remark}

Let us prove a similar statement for the case of $n\ge 3$.

\begin{theorem}\label{proportionalND}
Consider $n\ge 3$.
Let $f$ be a strongly generic $n$-ribbon surface in the space $C^{1,2,2,\ldots,2,1}_0([a,b],\r^3)$. Then the following two statements hold.

{\it $($i$)$} The surface $f$ has at most one degree of infinitesimal flexibility
$($i.e., all infinitesimal isometrically nontrivial flexions are proportional$)$.

{\it $($ii$)$} The surface $f$ is either finitely rigid or has one degree of finite flexibility.
\end{theorem}

\begin{proof} {\it $($i$)$} Let us assume the converse.
Suppose there are two non-proportional isometrically-nontrivial flexions $\Dd^1 f$ and $\Dd^2 f$.
By Proposition~\ref{iso-strong-iso}(i) both flexions are isometrically nontrivial at $(a,1)$
Hence there exists $\alpha$ such that the infinitesimal flexion
$$
\Dd f=\Dd^1 f- \alpha\Dd^2f
$$
is isometrically trivial at $(a,1)$. Since $\Dd^1 f$ and $\Dd^2 f$ are non-proportional, there exists
a point $(t,i)$ at which the flexion $\Dd f$ is isometrically nontrivial.
Hence by Proposition~\ref{iso-strong-iso}(i) the infinitesimal flexion $\Dd f$ is isometrically nontrivial at $(a,1)$.
We arrive at a contradiction.

{
\noindent
{\it $($ii$)$} By Theorem~\ref{proportionalND}(i) all infinitesimal flexions are proportional.
Hence $f$ has at most one degree of infinitesimal flexibility.
If it is zero, then the $f$ is infinitesimally rigid and hence it is finitely rigid.
}

Let $f$ has one degree of infinitesimal flexibility.
If $f$ does not have regular isometric deformations then $f$ is finite rigid.
If $f$ has a regular isometric deformation, then $f$ has a normalized isometric deformation.
Let us show that there exists at most one normalized isometric deformation of $f$.
Let $\{f^\lambda\}$ be a normalized isometric deformation of $f$.
As before we denote
$$
\Phi_i^\lambda=\(\Delta f_{i-1}^\lambda,\Delta f_i^\lambda\).
$$
Notice that for normalized isometric deformations we have:
$$
\Dd_{f^{\lambda}}(\Phi_0^\lambda (a))=(\Delta f_0^\lambda(a),\dot f_1^\lambda(a),\Delta f_1^\lambda).
$$
Therefore
$$
\Phi_0^\lambda (a)=\int\limits_{0}^{\lambda}(\Delta f_0^\mu(a),\dot f_1^\mu(a),\Delta f_1^\mu) d\mu.
$$
Hence $\Phi_0^\lambda (a)$ coincides for all normalized isometric deformation of $f$.
Therefore, by Proposition~\ref{inner} and Corollary~\ref{discr} for every $(t,i)$ and every parameter $\lambda$
the value
$$
\Phi_i^\lambda (t)
$$
is the same for all normalized isometric deformations.
Therefore, by Theorem~\ref{2-ribbon flex} every restriction of an arbitrary normalized isometric deformation $\{f^\lambda\}$
to the deformation of a 2-ribbon subsurface of $f$ does not depend on the choice of the normalized isometric deformation $\{f^\lambda\}$.
Hence, all normalized isometric deformations of $f$ coincide.
Therefore, $f$ has one degree of finite flexibility.
\end{proof}

\begin{remark}
The strongly genericity condition of Theorem~\ref{iso-strong-iso} is essential.
Let us illustrate this with a simple example of a 3-ribbon surfaces which is not strongly generic.
Consider
$$
f_0(t)=(t,1,0);
\quad
f_{1}(t)=(t,0,0);
\quad
f_{2}(t)=(t,0,1);
\quad
f_{3}(t)=(t,1,1);
$$
This surface has two distinct isometric deformations:
$$
\begin{array}{l}
\hbox{{\bf(i) }}
f^\alpha_0(t)=(t,1,0);
\quad
f^\alpha_{1}(t)=(t,0,0);
\quad
f^\alpha_{2}(t)=(t,0,1);
\quad
f^\alpha_{3}(t)=(t,\cos\alpha,1+\sin\alpha);
\\
\hbox{{\bf(ii) }}
f^\beta_0(t)=(t,1,0);
\quad
f^\beta_{1}(t)=(t,0,0);
\\
\qquad
f^\beta_{2}(t)=(t,\sin\beta, \cos\beta);
\quad
f^\beta_{3}(t)=\big(t,\sqrt{2}\sin(\beta+\frac{\pi}{4}),\sqrt{2}\cos(\beta+\frac{\pi}{4})\big);
\\
\end{array}
$$
The infinitesimal flexions defined by these isometric deformations are not proportional.
\end{remark}

\subsection{Flexibility of combined $n$-ribbon surfaces}

In this subsection we study finite and infinitesimal flexions of combined strongly generic semidiscrete surfaces.

As above, for an arbitrary semidiscrete surface $f=(f_0,f_1,\ldots,f_n)$ we denote
\begin{equation}\label{parts}
f^1=(f_0,f_1,\ldots, f_{n-1}),
\qquad
f^2=(f_1,\ldots, f_{n-1},f_n)
\quad
\hbox{and}
\quad
f^{12}=(f_1,\ldots, f_{n-1}).
\end{equation}

\subsubsection{Infinitesimal case}
We start with the infinitesimal case.

\begin{theorem}\label{combinedInf}{\bf (Infinitesimal flexibility of combined semidiscrete surfaces.)}
Let $n\ge 4$.
Consider a strongly generic $n$-ribbon semidiscrete surface $f$ in $C_0^{1,2,2,\ldots,2,1}([a,b],\r^3)$.
Let surfaces $f^1$ and $f^2$ defined by~$($\ref{parts}$)$ be infinitesimally flexible.
Then $f$ is infinitesimally flexible and has precisely one degree of infinitesimal flexibility.
\end{theorem}

\begin{proof}
Let $\Dd^1 f^1$ and $\Dd^2 f^2$ be isometrically nontrivial flexions of $f^1$ and $f^2$ respectively.
Since $f^1$ and $f^2$ are strongly generic $(n{-}1)$-ribbons surfaces for $n\ge 4$, Theorem~\ref{iso-strong-iso} can be applied.
By Theorem~\ref{iso-strong-iso}$($i$)$ the surfaces $f^1$ and $f^2$ are strongly isometrically nontrivial.
Hence by Theorem~\ref{iso-strong-iso}$($i$)$ in case $n>4$ or by Theorem~\ref{1-d-inf} (see Remark~\ref{proportional2D} above) in case $n=4$ the induced flexions
$\tilde \Dd^1 f^{12}$ and $\tilde \Dd^2 f^{12}$ are proportional, i.e,
there exists $\alpha$ such that
$$
\Dd^1 f^{12} = \alpha \Dd^2 f^{12}.
$$
Thus the surface $f$ has the combined infinitesimal flexion $\Dd f$ that induces $\Dd^1 f^1$ and $\alpha\Dd^2 f^2$.
This flexion is infinitesimally nontrivial, since the induced ones are infinitesimally nontrivial.
Hence $f$ has at least one degree of infinitesimal flexibility.

On the other hand by Theorem~\ref{iso-strong-iso}$($i$)$ the surface $f$ has at most one degree of infinitesimal flexibility.
Hence $f$ is infinitesimally flexible and has one degree of infinitesimal flexibility.
\end{proof}

\subsubsection{Finite case}
We start with the following general statement on reparametrisation of deformations.

\begin{proposition}\label{isom-deform}
Let $f$ be a strongly generic $n$-ribbon surface $($$n\ge 2$$)$ in the space $C^{1,2,2,\ldots,2,1}([a,b],\r^3)$.
Consider two regular isometric deformations $\gamma_1$ and $\gamma_2$ of $f$.
Then there exists a monotonous function $\xi$ such that
$\gamma_1(\lambda)=\gamma_2(\xi(\lambda))$ in some small neighborhood of $0$.
\end{proposition}

\begin{proof}
By Proposition~\ref{iso-strong-iso}(ii) there exists monotonous functions $\xi_1$ and $\xi_2$
such that $\gamma_1\circ \xi_1$ and $\gamma_2\circ \xi_2$ are normalized isometric deformations of $f$.
Hence by Theorem~\ref{proportionalND} in case $n\ge 3$ and Theorem~\ref{2-ribbon flex} (see Remark~\ref{proportional2D} above) in case $n=2$ we have
$$
\gamma_1\circ \xi_1=\gamma_2\circ \xi_2
$$
in some neighborhood of $0$. Set $\xi=\xi_2\circ \xi_1^{-1}$. The function $\xi$ is the monotonous function such that
$\gamma_1(\lambda)=\gamma_2(\xi(\lambda))$ in some small neighborhood of $0$.
\end{proof}

\begin{theorem}\label{combinedFin}{\bf (Finite flexibility of combined semidiscrete surfaces.)}
Let $n\ge 4$.
Consider a strongly generic $n$-ribbon semidiscrete surface $f$ in $C_0^{1,2,2,\ldots,2,1}([a,b],\r^3)$.
Let surfaces $f^1$ and $f^2$ defined by~$($\ref{parts}$)$ be finitely flexible.
Then $f$ is finitely flexible and has one degree of finite flexibility.
\end{theorem}

\begin{proof}
By Theorem~\ref{proportionalND}(ii) the surfaces $f^1$ and $f^2$ have one degree of finite flexibility.
Therefore, there exist unique normalized isometric deformations $\gamma_1$ and $\gamma_2$ of $f^1$ and $f^2$ respectively.
They induce two deformations $\tilde\gamma_1$ and $\tilde\gamma_2$ of the $(n{-}2)$-ribbon surface $f^{12}$.
By Proposition~\ref{isom-deform}, since $n-2\ge 2$,  these two deformations locally parameterize the same curve in $C_0^{1,2,2,\ldots,2,1}([a,b],\r^3)$,
i.e.,
in the segment $[-\varepsilon,\varepsilon]$ for some $\varepsilon>0$
there exists a locally increasing function $\xi$ such that $\tilde\gamma_1(\lambda)=\tilde\gamma_2(\xi(\lambda))$.

Now consider the deformation $\gamma$ of the surface $f$ inducing both isometric deformations
$\gamma_1$ for $f_1$ and $\gamma_2\circ \xi$ for $f_2$
in the segment $[-\varepsilon,\varepsilon]$ for some $\varepsilon>0$.
The deformation $\gamma$ is isometric, since the induced ones are isometric.
In addition, $\gamma$ is normalized, since its restriction to $f^1$ is a normalized isometric deformation.
Hence $f$ is finitely flexible (and not finitely rigid).
Therefore, by Theorem~\ref{iso-strong-iso}(ii) $f$ has one degree of finite flexibility.
\end{proof}

\begin{remark}
Note that the statements of Theorems~\ref{combinedInf} and~\ref{combinedFin} are no longer true for the case $n=3$.
On the one hand every infinitesimally flexible (and, therefore, finitely flexible) 3-ribbon surface
should satisfy a condition of Theorem~\ref{inf}, and, as it is easy to see, not every strongly generic 3-ribbon surface satisfies it.
Hence there are strongly generic rigid 3-ribbon surfaces.
On the other hand every 2-ribbon subsurface is weakly generic and hence it is  finitely (and, therefore, infinitesimally) flexible.
These two statements together contradict to the version of the statement of Theorems~\ref{combinedInf} for the case $n=3$.
\end{remark}

\subsection{An $n$-ribbon surface and its 3-ribbon subsurfaces}

Let us finally describe a relation between finite/infinitesimal flexibility of $n$-ribbon surfaces and finite/infinitesimal flexibility
of all 3-ribbon subsurfaces contained in them.

\begin{theorem}\label{n-ribbon-infini}
Let $n\ge 4$.
Consider a strongly generic $n$-ribbon surface $f$ in the space $C^{1,2,2,\ldots,2,1}_0([a,b],\r^3)$.
Then $f$ is infinitesimally flexible $($and has one degree of infinitesimal flexibility$)$ if and only if every
3-ribbon surface contained in the surface is infinitesimally flexible.
\end{theorem}

\begin{proof}
Let $f$ be infinitesimally flexible. Therefore, there exists an infinitesimal flexion $\Dd f$ that is
isometrically nontrivial.
Therefore, by  Proposition~\ref{iso-strong-iso}$($i$)$ the flexion $\Dd f$ is strongly isometrically nontrivial.
Hence all its 3-ribbon surfaces are isometrically nontrivially flexible.

\vspace{2mm}

Suppose now that all 3-ribbon subsurfaces in a strongly generic surface $f$ are infinitesimally flexible.
We prove that all $k$-ribbon surfaces are infinitesimally flexible for $k=3,4,\ldots,n$ by induction in $k$.

{\noindent {\it Base of induction.} The case $k=3$ is tautological.}

{\noindent {\it Step of induction.} The $k$-th statement follows from the $(k-1)$-th by Theorem~\ref{combinedInf}.}

Hence $f$ is infinitesimally flexible. Therefore, by Theorem~\ref{proportionalND}$($i$)$ $f$ has one degree of infinitesimal flexibility.
\end{proof}

For the finite flexibility we have the following.

\begin{theorem}\label{n-ribbon-fini}
Let $n\ge 4$.
Consider a strongly generic $n$-ribbon surface $f$ in the space  $C^{1,2,2,\ldots,2,1}_0([a,b],\r^3)$.
Then this surface is finitely flexible $($and has one degree of flexibility$)$
if and only if
every 3-ribbon surface contained in the surface is finitely flexible.
\end{theorem}

\begin{remark}
We think of this theorem as of a semidiscrete analogue to the
statement of the paper~\cite{BHS} on conjugate nets and all
$(3\times 3)$-meshes that they contain. In this paper we do not
study phenomena related to non-compactness and hence we restrict
ourselves to the case of compact $n$-ribbons surfaces.
\end{remark}

\begin{proof}
Let $f$ be finitely flexible. Therefore there exists a regular isometric deformation $\gamma$ of $f$.
Since $\gamma$ is regular we have $\Dd_{\gamma}f\ne 0$.
Since every finite flexion is infinitesimal flexion
we are in position to apply Proposition~\ref{iso-strong-iso}$($i$)$. We get that the flexion $\Dd_{\gamma}f$ is strongly isometrically nontrivial.
Hence the induced isometric deformations of all 3-ribbon surfaces have corresponding nontrivial finite flexions.
Therefore, all 3-ribbon surfaces contained in $f$ are finitely flexible.

\vspace{2mm}

Suppose that all 3-ribbon subsurfaces in a strongly generic surface $f$ are finitely flexible.
Let us prove that every $k$-ribbon surface in $f$ is finitely flexible for $k=3,4,\ldots,n$ by induction in $k$.

{\noindent {\it Base of induction.} The case $k=3$ is tautological.}

{\noindent {\it Step of induction.} The $k$-th statement follows from the $(k-1)$-th by Theorem~\ref{combinedFin}.}

Hence $f$ is finitely flexible. Therefore, by Theorem~\ref{proportionalND}$($ii$)$ $f$ has one degree of finite flexibility.
\end{proof}

\section{Isometric deformation of developable semidiscrete surfaces}

Suppose that all ribbons of a semidiscrete surface are
developable, i.e., the vectors $\dot f_i$, $\Delta f_i$, and $\dot
f_{i+1}$ are linearly dependent. We call such semidiscrete
surfaces {\it developable}. In this section we describe an
additional property for flexions of developable semidiscrete
surfaces. We start with 2-ribbon surfaces.

Recall that
$$
H_1=
\frac{(\dot f_1, \Delta \dot f_{0},\Delta f_{1})
+(\dot{f}_1, \Delta f_{0}, \Delta \dot f_{1})}
{(\dot{f}_1,\Delta f_{0},\Delta f_{1})}
$$
(as defined on page~\pageref{1111}).

\begin{proposition}\label{devel}
Consider a developable weakly generic 2-ribbon semidiscrete surface $f$ in $C^{1,2,1}([a,b],\r^3)$. Let
$$
\dot f_0(t) =a(t)\dot f_1(t)+b(t)\Delta f_{0}(t)
\quad \hbox{and} \quad
\dot f_2(t) =c(t)\dot f_1(t)+d(t)\Delta f_{1}(t).
$$
Then we have
$$
H_1(t)= d(t)-b(t).
$$
\end{proposition}

\begin{proof}
First, we have
\begin{align*}
(\dot f_1,\Delta \dot f_0,\Delta f_1)&=
(\dot f_1,\dot f_1-\dot f_0,\Delta f_1)=
-(\dot f_1,\dot f_0,\Delta f_1)=
-(\dot f_1,a\dot f_1+b\Delta f_{0},\Delta f_1)
\\
&=-b(\dot f_1,\Delta f_{0},\Delta f_1).
\end{align*}
Secondly, in a similar way we get
$$
(\dot f_1,\Delta f_0,\Delta \dot f_1)=d(\dot f_1,\Delta f_{0},\Delta f_1).
$$
Finally we have
$$
H_1=
\frac{(\dot f_1, \Delta \dot f_{0},\Delta f_{1})
+(\dot{f}_1, \Delta f_{0}, \Delta \dot f_{1})}
{(\dot{f}_1,\Delta f_{0},\Delta f_{1})}=d-b.
$$
This concludes the proof.
\end{proof}

This fact gives a surprising corollary concerning the flexion of a
2-ribbon developable surface.
Denote by $\alpha(t)$ the angle between $\Delta f_0(t)$ and $\Delta f_1(t)$.

\begin{corollary}\label{devel-cor}
Let $f$ be a weakly generic 2-ribbon developable surface in $C^{1,2,1}([a,b],\r^3)$.
Consider its isometric deformation $\gamma$.
Let us choose the parameter $\lambda$ of $\gamma$ such that
$\cos(\alpha (t_0))$ linearly depends on $\lambda$. Then for every
$t\in [a,b]$ the value $\cos(\alpha(t))$ linearly depends on $\lambda$.
\end{corollary}

\begin{proof}
First of all, notice that
$$
|\Delta f_0||\Delta f_1|\cos \alpha=\(\Delta f_0,\Delta f_1\)=\Phi,
$$
and hence
$$
\cos \alpha=\frac{\Phi}{|\Delta f_0||\Delta f_1|}.
$$

By Proposition~\ref{inner} and further by Proposition~\ref{devel} we have
$$
\D\Phi(t_1)=
\D\Phi(t_0)\cdot\exp{\left(\int\limits_{t_0}^{t_1} H_1(t)dt\right)}
=
\D\Phi(t_0)\cdot\exp{\left(\int\limits_{t_0}^{t_1}( d(t)- b(t))dt\right)}.
$$
Therefore, the ratio $\D\Phi(t_1)/\D\Phi(t_0)$ is a nonzero constant that depends entirely on the inner geometry of a 2-ribbon surface,
but not on its embedding in $\r^3$.
Therefore, the ratio
$$
\frac{\cos \alpha(t_1)}{\cos \alpha(t_0)}=\frac{\D\Phi(t_1)}{\D\Phi(t_0)}\frac{|\Delta f_0(t_0)||\Delta f_1(t_0)|}{|\Delta f_0(t_1)||\Delta f_1(t_1)|}
$$
is a nonzero constant that depends entirely on the inner geometry of a 2-ribbon surface but not on its embedding in $\r^3$ as well.
This implies the statement of the corollary.
\end{proof}

In fact, Corollary~\ref{devel-cor} implies a similar statement for an isometric deformation of a strongly generic $n$-ribbon developable surface.

\begin{corollary}
Consider a strongly generic finitely flexible $n$-ribbon developable surface of $f$ in $C^{1,2,1}([a,b],\r^3)$.
Let $\gamma$ be a nontrivial isometric deformation of $f$ $($i.e., $\D f\ne 0$$)$.
Then there exists a choice of the parameter $\lambda$ of the deformation $\gamma$,
such that for all $t\in[a,b]$ and $i\in\{1,\ldots, n{-}1\}$ all the cosines of the corresponding angles $\alpha_i(t)$
linearly depend on $\lambda$
$($here $\alpha_i(t)$ denotes the angle between $\Delta f_i(t)$ and $\Delta f_{i+1}(t)$$)$.
\qed
\end{corollary}

\vspace{5mm}

{\bf Acknowledgement.} The author is grateful to J.~Wallner for
constant attention to this work, A.~Weinmann for good remarks,
and the unknown reviewer for excellent comments and suggestions.
This work has been partially supported by the project ``Computational Differential
Geometry'' (FWF grant No.~S09209).

\vspace{5mm}

\end{document}